\documentclass[10pt,leqno,oneside,letterpaper]{amsart}
\usepackage{amssymb,amsmath,amsfonts,latexsym, amscd, amsthm, graphicx, url, pstricks, epsfig, pst-grad, pst-plot, amsbsy,epsf,graphics} 
\usepackage{color}
\usepackage[colorlinks,linkcolor=blue,citecolor=blue]{hyperref}
\usepackage[letterpaper,left=1truein, top=1truein]{geometry}

\usepackage{epstopdf}
\usepackage{pinlabel}

\makeatletter
\def\l@subsection{\@tocline{2}{0pt}{4pc}{5pc}{}}
\makeatother

\def\bysame{$\underline{\hskip.5truein}$}

\usepackage{epsfig}

\usepackage{caption,subcaption,graphicx,epsfig}
\usepackage{tikz-cd,calc} 
\usetikzlibrary{decorations.markings} 
\usetikzlibrary{arrows}
\tikzset{
  arrow/.pic={\path[tips,every arrow/.try,->,>=#1] (0,0) -- +(.1pt,0);},
  pics/arrow/.default={triangle 90}
}
\usetikzlibrary{knots}
\usetikzlibrary{decorations.markings}
\definecolor{light-gray}{gray}{0.92}
\definecolor{ultra-light-gray}{gray}{0.97}
\usepackage{mathrsfs}

\newtheorem{theorem}{Theorem}[section]
\newtheorem{lemma}[theorem]{Lemma}
\newtheorem{corollary}[theorem]{Corollary}
\newtheorem{proposition}[theorem]{Proposition}

\theoremstyle{definition}
\newtheorem{definition}[theorem]{Definition} 
 
\newtheorem{conjecture}[theorem]{Conjecture} 
 
\newtheorem{remark}[theorem]{Remark}

\newtheorem{example}[theorem]{Example}

\title[Order-detection of Slopes]{Order-detection of slopes on the boundaries of knot manifolds\footnotetext{2020 Mathematics Subject
Classification. Primary 06F15, 57M05, 57M99}}

\author[Steven Boyer]{Steven Boyer}
\thanks{Steven Boyer was partially supported by an NSERC grants RGPIN-9446-2013 and RGPIN-2018-06549}
\address{D\'epartement de Math\'ematiques, Universit\'e du Qu\'ebec \`a Montr\'eal, 201 President Kennedy Avenue, Montr\'eal, Qc., Canada H2X 3Y7.}
\email{boyer.steven@uqam.ca}
\urladdr{http://www.cirget.uqam.ca/boyer/boyer.html}

\author{Adam Clay}
\thanks{Adam Clay was partially supported by an NSERC grants RGPIN-2014-05465 and RGPIN-05343-2020.}
\address{Department of Mathematics, 420 Machray Hall, University of Manitoba, Winnipeg, MB, R3T 2N2. } 
\email{Adam.Clay@umanitoba.ca}
\urladdr{http://server.math.umanitoba.ca/~claya/}

\begin{document}

\begin{abstract}
Motivated by the $L$-space conjecture, we investigate various notions of order-detection of slopes on knot manifolds. These notions are designed to characterise when rational homology $3$-spheres obtained by gluing compact manifolds along torus boundary components have left-orderable fundamental groups and when a Dehn filling of a knot manifold has a left-orderable fundamental group. Our developments parallel the results of \cite{HRRW} in the case of Heegaard-Floer slope detection and of \cite{BGH1} in the case of foliation slope detection, leading to several conjectured structure theorems that connect relative Heegaard-Floer homology and the boundary behaviour of co-oriented taut foliations with the set of left-orders supported by the fundamental group of a $3$-manifold. The dynamics of the actions of $3$-manifold groups on the real line plays a key role in our constructions and proofs. Our analysis leads to conjectured dynamical constraints on such actions in the case where the underlying manifold is Floer simple.

\end{abstract}

\maketitle
 
\begin{center}
\today 
\end{center}

\section{Introduction} 
\label{sec: intro} 

The background motivation for the work in this paper is the $L$-space conjecture, which contends that a closed, connected, orientable, irreducible $3$-manifold is not an $L$-space if and only if it admits a co-oriented taut foliation and if and only if it has a left-orderable fundamental group; see \cite[Conjecture 1]{BGW} and  \cite[Conjecture 5]{Juh}. The conjecture has been verified in a number of situations, including all graph manifolds (\cite{BC, HRRW}), but remains widely open. 

A relative form of the conjecture, introduced in \cite{BC}, was the key to analysing the graph manifold case. The ideas involved are best illustrated by considering {\it knot manifolds}, i.e. compact, connected, orientable, irreducible $3$-manifolds $M \ne S^1 \times D^2$ with torus boundary. Each type of structure on $M$ (left-orders on its fundamental group, co-orientable taut foliations, Heegaard Floer homology) determines boundary data (sets of {\it detected} slopes) that encode the information needed to understand the behaviour of the structure with respect to gluing another knot manifold or a solid torus to $M$.  Consequently, slope detection comes in three forms: order-detection, foliation-detection, and Heegaard Floer-detection. A fourth notion, representation-detection (referring to representations of the fundamental group of $M$ with values in $\mbox{Homeo}_+(\mathbb R)$), was used in \cite{BC} to pass between left-orders and foliations. 

The ideas introduced in \cite{BC} readily extend from the Seifert fibred setting to general $3$-manifolds in the case of detection by foliations and Heegaard Floer homology (\cite{BGH1, RR, HRW}), but doing this for left-orders and representations requires a little more care. Here we focus on order-detection, though remark that there is a parallel development of representation-detection through the correspondence between orders on $\pi_1(M)$ and actions of $\pi_1(M)$ on the reals. This dynamical interpretation of left-orders is an essential ingredient in the constructions leading to the proof of our main gluing theorem, Theorem \ref{thm: gluing via detection}. It also suggests strong dynamical constraints on the fundamental groups of Floer simple knot manifolds. Indeed, Conjecture \ref{conj: floer simple implies bdry-cofinal} below can be rephrased to say that if $M$ is a Floer simple knot manifold, then an action of $\pi_1(M)$ on the reals has a fixed point if and only if the restriction of the action to $\pi_1(\partial M)$ has a fixed point. See \S \ref{sec: bdry cof and dynamics}.  

The initial difficulty in defining order-detection for general knot manifolds is that the most immediate notion, and the one sufficient for the Seifert fibred case, is inadequate in general. Nevertheless, this notion, which we call {\it weak order-detection}, is essential in underpinning our development of the {\it regular} and {\it strong} forms.  

The sets of weakly order-detected slopes, (regularly) order-detected slopes, and strongly order-detected slopes on the boundary of a knot manifold $M$ are denoted, respectively, by $\mathcal{D}_{ord}^{wk}(M)$,  $\mathcal{D}_{ord}(M)$ and $\mathcal{D}_{ord}^{str}(M)$. We will see that  
$$\mathcal{D}_{ord}^{str}(M) \subseteq \mathcal{D}_{ord}(M) \subseteq \mathcal{D}_{ord}^{wk}(M) \subseteq \mathcal{S}(M),$$
where $\mathcal{S}(M) \cong S^1$ denotes the set of slopes on $\partial M$, and that  $\mathcal{D}_{ord}(M)$ and $\mathcal{D}_{ord}^{wk}(M)$ are closed subsets of $\mathcal{S}(M)$, though $\mathcal{D}_{ord}^{str}(M)$ is not closed in general.   

Weak order-detection has an intuitive foundation: every left-order on $\pi_1(M)$ restricts to the subgroup $\pi_1(\partial M)$ yielding a left-order of $\mathbb{Z} \oplus \mathbb{Z} \subset \mathbb{R}^2$, and every such order determines a line in the plane that divides the subgroup $\mathbb{Z} \oplus \mathbb{Z}$ into positive and negative halves.  The slope of this line is the boundary data carried by weak order-detection. 

It turns out that this data alone is inadequate for the purpose of understanding when a manifold obtained by gluing together $3$-manifolds along incompressible boundary  tori has a left-orderable fundamental group; the necessary and sufficient conditions developed by Bludov and Glass for left-ordering an amalgam require that the amalgamating isomorphism satisfy a certain conjugacy invariance with respect to families of left-orders on each of the factors \cite[Theorem A]{BG}. We address this by defining (regular) order-detection as a strengthened form of weak order-detection which incorporates conjugacy invariance of the boundary data. Though weak order-detection  is \emph{a priori} very different from order-detection, our investigation suggests that they may be equivalent in the following sense:

\begin{conjecture}
 \label{conj: wk = reg intro} 
If $M$ is a knot manifold, then $\mathcal{D}_{ord}(M) = \mathcal{D}_{ord}^{wk}(M)$.
\end{conjecture}

We are able to verify this conjecture in a special case. 

\begin{theorem} 
\label{thm: some not weak means weak detd = detd}
If $\mathcal{D}_{ord}^{wk}(M) \ne  \mathcal{S}(M)$, then $\mathcal{D}_{ord}(M) = \mathcal{D}_{ord}^{wk}(M)$. 
\end{theorem}
Thus Conjecture \ref{conj: wk = reg intro}  is reduced to proving that if $\mathcal{D}_{ord}^{wk}(M) = \mathcal{S}(M)$, then $\mathcal{D}_{ord}(M) = \mathcal{S}(M)$.

From here, we investigate the behaviour of order-detection with respect to gluing and prove the following result, where it should be noted that we do not require that the identified slopes be rational. 

\begin{theorem}
\label{thm: gluing via detection}
Suppose that $M_1$ and $M_2$ are knot manifolds and $W = M_1 \cup_f M_2$, where $f: \partial M_1 \xrightarrow{\; \cong \;} \partial M_2$ identifies slopes $[\alpha_1] \in \mathcal{D}_{ord}(M_1)$ on $\partial M_1$ and $[\alpha_2] \in \mathcal{D}_{ord}(M_2)$ on $\partial M_2$. Then $\pi_1(W)$ is left-orderable. 
\end{theorem}
The utility of this theorem can be seen, for instance, in its use in the verification of the $L$-space conjecture for graph manifolds (\cite{BC}), in the proof that toroidal integer homology $3$-spheres have left-orderable fundamental groups (\cite{BGH1}), and in the proof that if a knot in the $3$-sphere admits a surgery with non-left-orderable fundamental group, then either the knot is cabled and the surgery is along the cabling slope or the $JSJ$ graph of the exterior of $K$ is an interval (\cite{BGH2}).  

Theorem \ref{thm: gluing via detection} leads to a characterisation of the order-detection of rational slopes that mirrors an analogous result for the Heegaard Floer notion of slope detection proved in \cite{RR, HRW} (i.e. $NLS$-detection for ``non-L-space" detection), and for the notion of foliation-detection investigated in \cite{BGH1}. Here, we use $N$ to denote the twisted $I$-bundle over the Klein bottle with rational longitude $\lambda_N$.  

\begin{theorem}  
\label{thm: topological def of order-detection} 
Suppose that $M$ is boundary-irreducible, $[\alpha]$ is a rational slope on $\partial M$, and $f: \partial N \to \partial M$ is a homeomorphism which identifies $[\lambda_N]$ with $[\alpha]$. Then $[\alpha]$ is order-detected if and only if $\pi_1(M \cup_f N)$ is left-orderable. 
\end{theorem} 

See \S \ref{subsec: general gluings} for generalisations of Theorems \ref{thm: gluing via detection} and \ref{thm: topological def of order-detection} to more general manifolds and gluings, e.g. manifolds with multiple torus boundary components in the first case and manifolds other than the twisted $I$-bundle over the Klein bottle in the second. 

It is conjectured (\cite[Conjecture 2.13 and Remark 7.5]{BGH1}) that the set of order-detected slopes on the boundary of a knot manifold coincides with both the set of foliation-detected slopes and the set of $NLS$-detected slopes. So, proceeding in analogy with the structure theorems on $NLS$-detection found in \cite{RR, HRW}, we expect that 
$\mathcal{D}_{ord}(M)$ is a connected subset of $\mathcal{S}(M)$ whose end-points are rational slopes encoded by the Turaev torsion of $M$. Further, we expect the converse to Theorem \ref{thm: gluing via detection} holds: 

\begin{conjecture}
 \label{conj: order gluing intro} 
Suppose that $M_1$ and $M_2$ are knot manifolds and $W = M_1 \cup_f M_2$, where $f: \partial M_1 \xrightarrow{\; \cong \;} \partial M_2$. If $\pi_1(W)$ is left-orderable, then $f$ identifies slopes $[\alpha_1] \in \mathcal{D}_{ord}(M_1)$ on $\partial M_1$ and $[\alpha_2] \in \mathcal{D}_{ord}(M_2)$ on $\partial M_2$.
\end{conjecture}

We have verified this conjecture in some cases.  Call a left-order $\mathfrak{o}$ on the fundamental group of a knot manifold $M$ {\it boundary-cofinal} if $\pi_1(\partial M)$ is not contained in any proper $\mathfrak{o}$-convex subset of $\pi_1(M)$. (See \S \ref{sec: lo groups}.) We show: 

\begin{theorem}
\label{thm: gluing via detection 2}
Suppose that $M_1$ and $M_2$ are knot manifolds such that every left-order on $\pi_1(M_2)$ is boundary-cofinal. Then $W = M_1 \cup_f M_2$ has a left-orderable fundamental group if and only if $f: \partial M_1 \xrightarrow{\; \cong \;} \partial M_2$ identifies slopes $[\alpha_1] \in \mathcal{D}_{ord}(M_1)$ on $\partial M_1$ and $[\alpha_2] \in \mathcal{D}_{ord}(M_2)$ on $\partial M_2$. 
\end{theorem}

In one of the main results of the paper, we provide a sufficient condition for all left-orders on $\pi_1(M)$ to be boundary-cofinal. 

\begin{theorem} 
\label{thm: some not weak implies all cofinal} 
If there is a slope on $\partial M$ which is not weakly order-detected, then each $\mathfrak{o} \in LO(M)$ is boundary-cofinal. 
\end{theorem}

Guided by what is known in the Heegaard Floer situation (\cite{RR, HRW}), we expect that knot manifolds $M$ satisfying $\mathcal{D}_{ord}^{wk}(M) \ne  \mathcal{S}(M)$ can be characterised as the family of \emph{Floer simple} knot manifolds. These are the knot manifolds $M$ for which there are at least two slopes in $\mathcal{S}(M)$ which are \emph{not} $NLS$-detected. 

\begin{conjecture}
 \label{conj: floer simple implies bdry-cofinal} 
A knot manifold $M$ is Floer simple if and only if $\mathcal{D}_{ord}^{wk}(M) \ne  \mathcal{S}(M)$ (and therefore each $\mathfrak{o} \in LO(M)$ is boundary-cofinal). 
\end{conjecture}

As examples, $L$-space knot exteriors are known to be Floer simple, and among such knots it is known that non-trivial torus knot exteriors have the property that $\mathcal{D}_{ord}^{wk}(M) \ne  \mathcal{S}(M)$ and that the same is true for certain pretzel knots, twisted torus knots and cables of such knots (\cite{BC, CW1, CW2}).  From Corollary \ref{cor: bdry cofinal iff no fixed pts}, we know that the exteriors of the closures of a $1$-bridge braids (\cite{Nie1}), which are $L$-space knots, and $(1,1)$ $L$-space knot exteriors (\cite{Nie2}) provide further examples. 

Lastly we introduce a strong form of order-detection for slopes on the boundary of a knot manifold $M$, defining it intrinsically in terms of relatively convex normal subgroups of $\pi_1(M)$ and showing that it is essentially equivalent to the left-orderability of the fundamental groups of the associated Dehn fillings: 

\begin{theorem}  
\label{thm: str det and df quot}  
Suppose that $[\alpha]$ is a rational slope on $\partial M$. Then $[\alpha] \in \mathcal{D}_{ord}^{str}(M)$ if and only if $\pi_1(M(\alpha))$ has a left-orderable quotient. 
\end{theorem}

Since the fundamental group of an orientable, irreducible $3$-manifold is left-orderable if it admits a left-orderable quotient (\cite[Theorem 1.1]{BRW}), we have the following corollary. 

\begin{corollary} 
\label {cor: str detect and dehn filling}
Suppose that $[\alpha]$ is a rational slope on $\partial M$. If $\pi_1(M(\alpha))$ is left-orderable, then $[\alpha]$ is strongly order-detected.  Conversely, 
if $[\alpha]$ is strongly order-detected and $M(\alpha)$ is irreducible, then $\pi_1(M(\alpha))$ is left-orderable. 
\qed
\end{corollary}

We remark that there are at most three reducible Dehn fillings of an irreducible knot manifold (\cite{GLu}), so strong order-detection does indeed give an accurate picture of which fundmental groups of manifolds obtained via Dehn filling are left-orderable.

\subsection{Organisation of the paper.}

We review some basic notions and properties of left-ordered groups in \S \ref{sec: lo groups} and the construction of the dynamic realisation of a left-order on a countable group in \S \ref{sec: dyn reals} including a discussion of the family of left-orders obtained from the associated action on the real line. Section \ref{sec: lo and 3mflds} contains background material on slopes and left-orderable $3$-manifold groups. The notions of weak order-detection, (regular) order-detection, and strong order-detection are developed in \S \ref{sec: weak}, \S \ref{sec: reg}, and \S \ref{sec: str} respectively. Section \ref{sec: bdry cof and dynamics} translates the order-theoretic notion of boundary-cofinality into a condition on actions on the real line.

\section{Left-orderable groups} 
\label{sec: lo groups}
In this section we review some basic notions and properties of left-ordered groups. 

\subsection{Generalities}
\label{subsec: generalities} 
A {\it left-order} $\mathfrak{o}$ on a non-trivial group $G$ can be defined by a total order $<_{\mathfrak{o}}$ on $G$ that is invariant under left-multiplication, or by a semigroup $P(\mathfrak{o}) \subset G$ satisfying $G = P(\mathfrak{o}) \sqcup \{1\} \sqcup P(\mathfrak{o})^{-1}$.  (Here we use $A^{-1}$ to denote $\{ g^{-1} \; | \; g \in A\}$ for a subset $A \subseteq G$.)  One can see the equivalence of these notions by noting that $P(\mathfrak{o})  = \{ g \in G \mid 1 <_{\mathfrak{o}} g \} \subset G$, the {\it positive cone} of $\mathfrak{o}$, satisfies $G = P(\mathfrak{o}) \sqcup \{1\} \sqcup P(\mathfrak{o})^{-1}$.  Conversely, any semigroup $S \subset G$ such that $G = S \sqcup \{1 \} \sqcup S^{-1}$ determines a left-order $\mathfrak{o}$ and its corresponding left-invariant total order $<_{\mathfrak{o}}$ of $G$ according to $g<_{\mathfrak{o}}h$ if and only if $g^{-1}h \in S$.  A group is called {\it left-orderable} if it is non-trivial and admits a left-order. 

Let $\mathfrak{o}$ be a left-order on a group $G$. An $\mathfrak{o}$-{\it convex} subset of $G$ is a subset $A \subseteq G$ such that if $k, h \in A$ and $g \in G$ satisfy $k <_{\mathfrak{o}} g <_{\mathfrak{o}} h$, then $g \in A$. For $h, k \in G$ we define $\mathfrak{o}$-convex subsets $(h, k), [h, k], (h, \infty), $ etc. in the usual way. 

The $\mathfrak{o}$-{\it convex hull} of a subset $A$ of $G$ is 
$$C(A) = \{g \in G \; | \; \mbox{ there are $a_1, a_2 \in A$ such that } a_1 \leq_{\mathfrak{o}}  g \leq_{\mathfrak{o}}  a_2 \}.$$ 
We say that a subset $A$ of $G$ is $\mathfrak{o}$-{\it cofinal} if its $\mathfrak{o}$-convex hull is all of $G$, and call an element $g \in G$ $\mathfrak{o}$-cofinal if the cyclic subgroup $\langle g \rangle \subset G$ is $\mathfrak{o}$-cofinal.   We will often say ``cofinal" or ``convex hull" for short when the left-order ${\mathfrak{o}}$ is understood. 

The set of left-orders on a group $G$ is denoted by $LO(G)$.  Endowing $LO(G)$ with the Sikora topology \cite{Si} yields a compact, Hausdorff, totally disconnected space which is metrisable when $G$ is countable. Setting 
$$P(g \cdot \mathfrak{o}) = g P(\mathfrak{o}) g^{-1}$$
determines an action of $G$ on $LO(G)$ by homeomorphisms.
 
A {\it normal family} of left-orders on $G$ is a $G$-invariant subset $\mathcal{N}$ of $LO(G)$. For instance, the orbit  
$$\mathcal{O}(\mathfrak{o}) = \{g \cdot \mathfrak{o} \; | \; g \in G\} \subseteq LO(G)$$
of $\mathfrak{o} \in LO(G)$ is a normal family, as is its closure $\overline{\mathcal{O}(\mathfrak{o})}$ in $LO(G)$. 

There is a continuous involution on $LO(G)$ given by taking {\it opposites}: For $\mathfrak{o} \in LO(G)$, $\mathfrak{o}^{op}$ is the left-order defined by 
$$g <_{\mathfrak{o}^{op}} h \Leftrightarrow h <_{\mathfrak{o}} g$$
Equivalently, 
$$P(\mathfrak{o}^{op}) = P(\mathfrak{o})^{-1}$$

\subsection{Left-orders on $\mathbb Z^2$} 
\label{subsec: lo z2} 
A key example for us is the case $G = \mathbb Z^2$. The basic properties of positive cones imply that given a left-ordering $\mathfrak{o}$ of $\mathbb Z^2$, there is a line $L({\mathfrak{o}}) \subset H_1(\mathbb Z^2; \mathbb R) = \mathbb Z^2  \otimes \mathbb R \cong \mathbb R^2$ uniquely determined by the fact that all elements of $\mathbb Z^2$ which lie to one side of it are $\mathfrak{o}$-positive and all elements lying to the other side are $\mathfrak{o}$-negative. (See \cite[Lemma 3.3]{CR1}.) We say that $L({\mathfrak{o}}) $ has {\it rational slope} if $L_0 = L(\mathfrak{o}) \cap \mathbb Z^2 \cong \mathbb Z$, in which case $L_0$ is $\mathfrak{o}$-convex. Otherwise we say that it has {\it irrational slope}.

Each line $L$ in $\mathbb R^2$ through $(0,0)$ is realised as $L({\mathfrak{o}})$ for some $\mathfrak{o} \in LO(\mathbb Z^2)$. Further, there are exactly two orders realising $L$ when it  has irrational slope and four when it has rational slope. More precisely, suppose that $L = L({\mathfrak{o}})$. If
\vspace{-.2cm}
\begin{itemize}

\item $L$ has irrational slope, $\{\mathfrak{o}, \mathfrak{o}^{op}\}$ are the two left-orders on $\mathbb Z^2$ which realise $L$.

\vspace{.2cm} \item $L$ has rational slope and $L_0 = L \cap \mathbb Z^2 \cong \mathbb Z$, let $\mathfrak{o}^* \in LO(\mathbb Z^2)$ be defined by 
$$\left\{ \begin{array}{c} 
\mathfrak{o}^*|_{L_0} = (\mathfrak{o}|_{L_0})^{op} \\ 
P(\mathfrak{o}^*)  \setminus L_0= P(\mathfrak{o})   \setminus L_0
\end{array} \right.$$
Then $\{\mathfrak{o}, \mathfrak{o}^{op}, \mathfrak{o}^*, (\mathfrak{o}^*)^{op}\}$ are the four left-orders on $\mathbb Z^2$ which realise $L$. 
\end{itemize}
For every $\mathfrak{o} \in LO(\mathbb Z^2)$  the set of $\mathfrak{o}$-cofinal elements is precisely $\mathbb Z^2 \setminus L(\mathfrak{o})$. 

\subsection{Convex subgroups} 
\label{subsec: convex subgroups} 
If $\mathfrak{o}$ is a left-order on a group $G$, it is straightforward to show that a subgroup $C \subset G$ is $\mathfrak{o}$-convex if and only if the inequality $1 <_\mathfrak{o} g <_\mathfrak{o} h$ for $g \in G$ and $h \in C$ implies that $g \in C$. Similarly, $C$ is $\mathfrak{o}$-convex if and only the inequality $h <_\mathfrak{o} g <_\mathfrak{o} 1$ for $g \in G$ and $h \in C$ implies that $g \in C$. 

The convexity condition has some immediate consequences:
\begin{itemize}

\item proper convex subgroups $C$ of $G$ are of infinite index, for if $g \not \in C$ then $g^n \not \in C$ for all non-zero integers $n$; 

\vspace{.2cm} \item proper convex subgroups $C$ of $G$ are $\mathfrak{o}$-bounded, for if $g \in P(\mathfrak{o}) \setminus C$, then $C \subset [g^{-1}, g]$. 

\vspace{.2cm} \item the set of $\mathfrak{o}$-convex subgroups of a group $G$ is linearly ordered by inclusion, for if $C, D$ are $\mathfrak{o}$-convex and $d  \in P(\mathfrak{o}) \cap (D \setminus C)$, then $C \subset [d^{-1}, d] \subseteq D$.

\end{itemize}

\begin{lemma} 
\label{lemma: convex and cosets}
Let $C$ be a subgroup of a group $G$ with left-order $\mathfrak{o}$.

$(1)$ If $C$ is $\mathfrak{o}$-convex, then $P(\mathfrak{o}) \setminus C$ is a union of right $C$-cosets.

$(2)$ $C$ is $\mathfrak{o}$-convex if and only if $P(\mathfrak{o}) \setminus C$ is a union of left $C$-cosets. 
\end{lemma}

\begin{proof}
First suppose that $C$ is $\mathfrak{o}$-convex. If $g \in P(\mathfrak{o}) \setminus C$ and $c \in C$, then $cg >_{\mathfrak{o}} 1$ as otherwise $1 <_{\mathfrak{o}} g <_{\mathfrak{o}} c^{-1}$ so that $g \in C$, contrary to our choices. Similarly $gc >_\mathfrak{o} 1$ as otherwise $c <_\mathfrak{o} g^{-1} <_\mathfrak{o} 1$, which would imply $g \in C$. Then both $Cg$ and $gC$ are contained in $P(\mathfrak{o}) \setminus C$, which proves (1) and the forward direction of (2). 

For the reverse direction of (2), suppose that $P(\mathfrak{o}) \setminus C$ is a union of left $C$-cosets and that $1 <_\mathfrak{o} g <_\mathfrak{o} c$, where $g \in G$ and $c \in C$. If  $g \not \in C$, then $gC \subset P(\mathfrak{o}) \setminus C$ and $g^{-1}C = (g^{-1}c) C \subset P(\mathfrak{o}) \setminus C$. But then $g, g^{-1} \in P(\mathfrak{o})$, which is impossible. Thus $g \in C$ and therefore $C$ is $\mathfrak{o}$-convex, which completes the proof. 
\end{proof}

Consequently,

\begin{proposition} 
\label{prop: still convex}
Suppose that $\mathfrak{o}_1 \in LO(G)$ and $C$ is $\mathfrak{o}_1$-convex.

$(1)$ If $\mathfrak{o}_2 \in LO(G)$ and $P(\mathfrak{o}_1) \setminus C = P(\mathfrak{o}_2) \setminus C$, then $C$ is $\mathfrak{o}_2$-convex.

$(2)$ If $\mathfrak{o}_0 \in LO(C)$, then $P(\mathfrak{o}_0) \sqcup \big(P(\mathfrak{o}_1) \setminus C\big)$ is the positive cone of a left-order on $G$. 
\qed
\end{proposition}

\begin{definition}
{\rm 
Let $G$ be a group and $\mathfrak{o}_1, \mathfrak{o}_2 \in LO(G)$. If $C$ is a proper $\mathfrak{o}_1$-convex subgroup of $G$ and $P(\mathfrak{o}_1) \setminus C = P(\mathfrak{o}_2) \setminus C$ we say that $\mathfrak{o}_1$ is {\it convexly related} to $\mathfrak{o}_2$ and that $\mathfrak{o}_1$ and $\mathfrak{o}_2$ differ by a {\it convex swap}.  
}
\end{definition}

Since the set of $\mathfrak{o}$-convex subgroups of a group $G$ is linearly ordered by inclusion, Proposition \ref{prop: still convex} implies that being convexly related is an equivalence relation. 

It also follows from Lemma \ref{lemma: convex and cosets} that if $g < h$, $gC \neq hC$ and $c_1, c_2 \in C$, then $gc_1 < hc_2$.  To see this, note that $g^{-1}h \in P \setminus C$ so that $g^{-1}h$ lies in a right coset of $C$ consisting entirely of positive elements, whence $c_1^{-1}g^{-1}h$ is positive.  Similarly since $c_1^{-1}g^{-1}h \in P \setminus C$ and $P \setminus C$ is a union of left $C$-cosets, we arrive at $1 < c_1^{-1}g^{-1}hc_2$ and so $gc_1 < hc_2$. We therefore arrive at the following standard result:

\begin{proposition} \cite[Proposition 2.1.3]{KM}
\label{prop: induced order}
Suppose that $\mathfrak{o} \in LO(G)$ and $C$ is $\mathfrak{o}$-convex. Then $\mathfrak{o}$ induces a $($left$)$ $G$-invariant total order $\overline{\mathfrak{o}}$ on $G/C$ as follows $:$ $gC <_{\overline{\mathfrak{o}}} hC$ if and only if  $g^{-1}h \not \in C$ and $g <_{\mathfrak{o}} h$. 
\qed
\end{proposition}

The following proposition can be seen as a converse to Proposition \ref{prop: induced order}. 
 
\begin{proposition}   \cite[Corollary 5.1.4]{KM}
\label{prop: action yields lo}
Suppose that a non-trivial group $G$ acts by order-preserving permutations on a totally ordered set $(E, <_E)$. Fix $x \in E$ and let $\mbox{Stab}(x) \leq G$ be the stabiliser of $x$. Given a left-order $\mathfrak{o}_0$ on $\mbox{Stab}(x)$, there is a left-order $\mathfrak{o}_x$ on $G$ extending $\mathfrak{o}_0$ given by
$$g_1 <_{\mathfrak{o}_x} g_2 \hbox{ if and only if } \left\{ 
\begin{array}{cl} 
1 <_{\mathfrak{o}_0} g_1^{-1}g_2 & \hbox{when } g_1^{-1}g_2 \in \mbox{Stab}(x) \\ 
& \\
g_1 \cdot x <_E g_2 \cdot x & \hbox{when } g_1^{-1}g_2 \not \in \mbox{Stab}(x)
\end{array}
\right.$$
Further, if $g_1, g_2 \in G$ then $g_1 \leq_{\mathfrak{o}_x} g_2$ implies that $g_1 \cdot x \leq_E g_2 \cdot x$. Consequently, $\mbox{Stab}(x)$ is $\mathfrak{o}_x$-convex. 
\end{proposition}

\begin{definition} 
\label{def: rel convex}
A subgroup $C \subset G$ is {\it relatively convex} if there exists an ordering $\mathfrak{o}$ of $G$ such that $C$ is $\mathfrak{o}$-convex. 
\end{definition}

Since relatively convex subgroups of $G$ are closed under taking roots, not all subgroups of all left-orderable groups have this property. Conversely, a nontrivial left-orderable group that admits no proper, relatively convex subgroup must be an abelian group of rank one \cite[Proposition 5.1.9]{KM}.

\begin{proposition} \cite[Proposition 5.1.10]{KM}
\label{prop: relatively convex}
Suppose that $G$ is a left-orderable group. Then an arbitrary intersection of relatively convex subgroups of $G$ is relatively convex.
\end{proposition}

\section{Dynamic realisations and associated left-orders}  
\label{sec: dyn reals}

Throughout this section we take $G \ne \{1\}$ to be a countable group endowed with a left-order $\mathfrak{o}$. We review the construction of the dynamic realisation of the pair $(G, \mathfrak{o})$ and analyse the families of left-orders obtained from the associated action on the real line.

\subsection{Tight embeddings of left-orderable groups} 
An $\mathfrak{o}${\it -gap} in $G$ is a pair of elements $h, k \in G$ such that $h < k$ and $(h, k) = \emptyset$. 

Consider an order-preserving embedding $t: (G, 1) \to (\mathbb R, 0)$ which is {\it tight}: if $-\infty \leq a < b \leq \infty$ and the open  interval $(a, b)$ is contained in $\mathbb R \setminus t(G)$, then there is a gap $\{h, k\}$  in $G$ such that $(a, b) \subseteq (t(h), t(k))$. In other words, the only gaps in $t(G) \subset \mathbb R$ are ones which come from $\mathfrak{o}$-gaps in $G$. Tight embeddings exist; the standard order-preserving embedding of $G$ into $\mathbb R$ (cf. \cite[page 29, \S 2.4]{CR}) is tight. They are also well-defined up to composition with an element of $\mbox{Homeo}_+(\mathbb R)$ (cf. \S \ref{subsec: dyn real}). It is an immediate consequence of the definition that tight embeddings are unbounded above and below. 

Partition the reals into three subsets
$$\mathbb R = t(G) \sqcup (\overline{t(G)} \setminus t(G)) \sqcup (\mathbb R \setminus \overline{t(G)})$$ 
We refer to the points of $\overline{t(G)} \setminus t(G)$ as {\it ideal points} of $\mathfrak{o}$. It follows from the definition of a tight embedding that
\vspace{-.cm}
\begin{itemize}

\item each ideal point of $\mathfrak{o}$ is both a right limit point and left limit point of $t(G)$; 

\vspace{.2cm} \item the gaps in $G$  correspond bijectively to the components of $\mathbb R \setminus \overline{t(G)}$ under the map which identifies a gap $\{h, k\}$ with the interval $(t(h), t(k))$. 

\end{itemize}
These properties are useful in supplying the details of the construction of a dynamic realisation of the pair $(G, \mathfrak{o})$ below. 

\subsection{Dynamic realisations} 
\label{subsec: dyn real}
Given a tight embedding $t: G \to \mathbb R$, there is an associated action of $G$ on $\mathbb R$ :      
\begin{itemize}

\item Begin with the free action of $G$ on the subspace $t(G)$ of $\mathbb R$ : $g \cdot t(g') = t(gg')$. The tightness of $t$ implies that the action is by homeomorphisms.

\vspace{.2cm} \item Extend the action over the closure $\overline{t(G)}$ of $t(G)$: If $g \in G$ and $\{t(g_n)\}$ is a convergent sequence  
of elements of $t(G)$, the tightness of $t$ guarantees that $\{t(gg_n)\}$ is also a convergent sequence. The correspondence 
$$g \cdot \lim_n t(g_n) = \lim_n t(gg_n)$$ 
determines an action of $G$ on $\overline{t(G)}$ by homeomorphisms. 

\vspace{.2cm} \item Extend the homeomorphism of $\overline{t(G)}$ determined by $g \in G$ to $\mathbb R \setminus \overline{t(G)}$  by convex interpolation: First note that the components of $\mathbb R \setminus \overline{t(G)}$ are open intervals whose endpoints are elements of $\overline{t(G)}$. Therefore if $(x, y)$ is a component of $\mathbb R \setminus \overline{t(G)}$ with $x, y \in \overline{t(G)}$, we may define
$$g \cdot((1-s)x + s y) = (1-s)(g \cdot x) + s(g \cdot y)$$
where $0 \leq s \leq 1$.  This extension determines a homeomorphism $\rho_{\mathfrak{o}}(g)$ of $\mathbb R$ and the correspondence
$$\rho_{\mathfrak{o}}: G \to \mbox{Homeo}(\mathbb R)$$
is a homomorphism. Since $G$ acts freely on $t(G)$, the action is faithful and since the left action of $G$ on $t(G)$ is $\mathbb R$-order-preserving, the image of $\rho_{\mathfrak{o}}$ is contained in $\mbox{Homeo}_+(\mathbb R)$. Further, since
$$g <_{\mathfrak{o}} h \Leftrightarrow \rho_{\mathfrak{o}}(g)(0) = t(g) < t(h) = \rho_{\mathfrak{o}}(h)(0),$$
$\rho_{\mathfrak{o}}$ determines $\mathfrak{o}$ by considering the orbit of $0$. 
\end{itemize}
The fact that $t(G)$ is unbounded above and below implies  that there is no $x \in \mathbb R$ which is fixed by every element of $G$. In other words, the action has no global fixed points. Such an action is called {\it non-trivial}. 
 
The representation $\rho_{\mathfrak{o}}: G \to \mbox{Homeo}_+(\mathbb R)$ depends only on the embedding $t$. Further, if $t': (G, 1) \to  (\mathbb R, 0)$ is another tight embedding, the method used to construct $\rho_{\mathfrak{o}}$ can be used to show that the correspondence 
$t(G) \to t'(G), t(g) \mapsto t'(g)$, extends to an orientation-preserving homeomorphism $f$ of the reals in such a way that if $\rho_{\mathfrak{o}}': G \to \mbox{Homeo}_+(\mathbb R)$ is the homomorphism associated to $t'$, then $\rho_{\mathfrak{o}}' = f \circ \rho_{\mathfrak{o}} \circ f^{-1}$. Thus $\rho_{\mathfrak{o}}$ is well-defined up to conjugation in $\mbox{Homeo}_+(\mathbb R)$.  

Summarising, 

\begin{proposition}
\label{prop: dyn realisation} 
Given a countable group $G$ with left-order $\mathfrak{o} \in LO(G)$, there is a faithful representation $\rho_{\mathfrak{o}}: G \to \mbox{Homeo}_+(\mathbb R)$, well-defined up to conjugation in $\mbox{Homeo}_+(\mathbb R)$, which induces a non-trivial action of $G$ on $\mathbb R$. Further,  $\rho_{\mathfrak{o}}$ determines $\mathfrak{o}$ via 
$$g <_{\mathfrak{o}} h \Leftrightarrow \rho_{\mathfrak{o}}(g)(0)  < \rho_{\mathfrak{o}}(h)(0).$$
\end{proposition}

The representation $\rho_{\mathfrak{o}}: G \to \mbox{Homeo}_+(\mathbb R)$ is referred to as the {\it dynamic realisation} of $(G, \mathfrak{o})$.

\subsection{Left-orders associated to dynamic realisations}  
\label{subsec: left-orders from dyn reals}
We noted in Proposition \ref{prop: dyn realisation} that a left-order $\mathfrak{o} \in LO(G)$ is determined by the natural order on the orbit of $0$ under the dynamic realisation $\rho_{\mathfrak{o}}$ of $\mathfrak{o}$. Here we show how the orbit of each $x \in \mathbb R$ under $\rho_{\mathfrak{o}}$ determines a convexly related family of left-orders $\mathfrak{O}_x$.

For $x \in \mathbb R$ we denote the stabiliser $\{ g \in \pi_1(M) \; | \; \rho_{\mathfrak{o}}(g)(x) = x\}$ of $x$ under the $\rho_{\mathfrak{o}}$-action by $\mbox{Stab}_{\rho_{\mathfrak{o}}}(x)$. 

\begin{lemma} 
\label{l: stabilisers} 
If $x \in \mathbb R$, then $\hbox{Stab}_{\rho_{\mathfrak{o}}}(x) \ne \{1\}$ implies that $x$ is an ideal point of $\mathfrak{o}$. 
\end{lemma}

\begin{proof}
It is clear that if $x \in t(G)$, then $\hbox{Stab}_{\rho_{\mathfrak{o}}}(x) = \{1\}$. If, on the other hand, $x$ is contained in a gap $(t(g), t(h))$ of $\mathfrak{o}$, then by  construction any $k \in \hbox{Stab}_{\rho_{\mathfrak{o}}}(x)$ is the identity on $(t(g), t(h))$. In particular, $k \in \hbox{Stab}_{\rho_{\mathfrak{o}}}(t(g)) = \{1\}$, which completes the proof. 
\end{proof}

The restriction of $\mathfrak{o}$ to $\hbox{Stab}_{\rho_{\mathfrak{o}}}(x)$ determines a left-order $\mathfrak{o}_x \in LO(G)$ (cf. Proposition \ref{prop: action yields lo}) with respect to which $\hbox{Stab}_{\rho_{\mathfrak{o}}}(x)$ is convex. Further, each orbit $\mathcal{O}_{\rho_{\mathfrak{o}}}(x) = \{\rho_{\mathfrak{o}}(g)(x) \; | \; g \in G\} \subset \mathbb R$ is a left $G$-set canonically identifiable with the cosets $E_x := G/\mbox{Stab}_{\rho_{\mathfrak{o}}}(x)$ under the map $g\mbox{Stab}_{\rho_{\mathfrak{o}}}(x) \mapsto \rho_{\mathfrak{o}}(g)(x)$. It is easy to see that this identification also takes the total order $\overline{\mathfrak{o}}_x$ on $E_x$ given by Proposition \ref{prop: induced order} to the natural order on $\mathcal{O}_{\rho_{\mathfrak{o}}}(x)$ coming from $\mathbb R$. 

\begin{definition}
\label{def: O_x}
{\rm 
Let $\mathfrak{O}_x$ denote the set of left-orders on $G$ obtained, vis-a-vis Proposition \ref{prop: action yields lo}, from the action of $G$ on $E_x$ 
equipped with its natural total order $\overline{\mathfrak{o}}_x$, and a choice of left-order on $\mbox{Stab}_{\rho_{\mathfrak{o}}}(x)$.
}
\end{definition}
Proposition \ref{prop: action yields lo} shows that $\mbox{Stab}_{\rho_{\mathfrak{o}}}(x)$ is $\mathfrak{o}$-convex for each $\mathfrak{o} \in \mathfrak{O}_x$ and so it is a family of convexly related left-orders. Further, $\mathfrak{O}_x$ always contains the left-order $\mathfrak{o}_x$ determined by the orders $\mathfrak{o}|_{{\tiny \hbox{Stab}}_{\rho_\mathfrak{o}}(x)}$ and 
 $\overline{\mathfrak{o}}_x$. If $\mbox{Stab}_{\rho_{\mathfrak{o}}}(x)$ is trivial, for instance if $x$ is not an ideal point of $\mathfrak{o}$, then  $\mathfrak{O}_x = \{\mathfrak{o}_x\}$.  However if $\mbox{Stab}_{\rho_{\mathfrak{o}}}(x)$ is non-trivial, then $\mathfrak{O}_x$ is uncountable unless $\mbox{Stab}_{\rho_{\mathfrak{o}}}(x)$ is a so-called ``Tararin group'' (\cite[Theorem 5.2.1]{KM}), in which case it has order $2^n$ for some $n>0$. Proposition \ref{prop: action yields lo} shows that for any 
$\mathfrak{o}_x' \in \mathfrak{O}_x$, 
$$g \leq_{\mathfrak{o}_x'} h \; \Rightarrow \; \rho_{\mathfrak{o}}(g)(x) \leq \rho_{\mathfrak{o}}(h)(x)$$ 
with equality on the right hand side if and only if $g^{-1}h \in \mbox{Stab}_{\rho_{\mathfrak{o}}}(x)$, so whether or not $g \leq_{\mathfrak{o}_x'} h$ holds is determined by the chosen ordering of $\mbox{Stab}_{\rho_{\mathfrak{o}}}(x)$. 

\begin{proposition}
\label{prop: relation to orbit}
Let $\mathfrak{o} \in LO(G)$ and $\rho_{\mathfrak{o}}$ an associated dynamical realisation. 

$(1)$ If $x \in \mathbb R$ is not an ideal point of $\mathfrak{o}$ with respect to $\rho_\mathfrak{o}$, then $\mathfrak{O}_x = \{\mathfrak{o}_x\}  \subset \mathcal{O}(\mathfrak{o})$.

$(2)$ If $x \in \mathbb R$ is an ideal point of $\mathfrak{o}$, then $\mathfrak{O}_x \cap \overline{\mathcal{O}(\mathfrak{o})} \ne \emptyset$.
\end{proposition}

We divide the proof of Proposition \ref{prop: relation to orbit} into three lemmas. The first proves part (1) of the proposition. 

\begin{lemma} 
\label{lemma: conjugate orders 1}
$\;$ 

$(1)$ $\mathfrak{o}_{t(g)} = g \cdot \mathfrak{o}$ for each $g \in G$. 

$(2)$ If $x \in (t(g), t(h))$, where $\{g, h\}$ is a gap of $\mathfrak{o}$, then $\mathfrak{o}_{x} = g \cdot \mathfrak{o}$. 

\end{lemma}

\begin{proof}
(1) Since the stabiliser of $t(g)$ is trivial, we have $1 <_{\mathfrak{o}_{t(g)}} h$ if and only if $t(g) < \rho_{\mathfrak{o}}(h)(t(g)) = t(hg)$. The latter holds if and only if $g <_{\mathfrak{o}} hg$, or equivalently, $1 <_{\mathfrak{o}} g^{-1} h g$. Thus $P(\mathfrak{o}_{t(g)}) = gP(\mathfrak{o}) g^{-1}$, which completes the proof.

(2) The orbit of $(t(g), t(h))$ under $G$ is $\{(t(kg), t(kh)) \; | \; k \in G\}$, a family of disjoint open intervals in obvious bijection with $G$ whose union is $\mathbb R \setminus \overline{t(G)}$. It is easy to see that $\mathfrak{o}_{x} = \mathfrak{o}_{t(g)}$, so the result follows from (1). 
\end{proof}

Recall that the power set $\{0, 1\}^X$ of any set $X$ can be topologised by equipping $\{0, 1\}$ with the discrete topology, and $\{0, 1\}^X$ with the resulting product topology.  This makes $\{0, 1\}^X$ into a compact, Hausdorff, totally disconnected space. In the resulting topology, a sequence of subsets $\{S_n\}$ of $X$ converges to $S \subset X$ if, for every finite $F \subset X$ there exists $n \in \mathbb{N}$ such that $S \cap F = S_{n_0} \cap F$ for all $n_0 \geq n$. The Sikora topology on $LO(G)$ is obtained by identifying $LO(G)$ with its image under the embedding $LO(G) \to \{0, 1\}^G, \mathfrak{o} \mapsto P(\mathfrak{o})$.  

\begin{lemma} \label{l: orders associated to ideal points 1}
Suppose that $x$ is an ideal point of $\mathfrak{o}$ and let $\{g_n\}$ be a sequence of elements of $G$ such that $\lim_n t(g_n) = x$. Then 
$$P(\mathfrak{o}_x) \setminus \hbox{Stab}_{\rho_{\mathfrak{o}}}(x) = \lim_n \; P(\mathfrak{o}_{t(g_n)}) \setminus \hbox{Stab}_{\rho_{\mathfrak{o}}}(x).$$
\end{lemma}

\begin{proof}
Let $g \in G \setminus \hbox{Stab}_{\rho_{\mathfrak{o}}}(x)$ and observe that as 
$$\left\{ \begin{array}{l} 
g <_{\mathfrak{o}_x} 1 \; \Rightarrow \;  \rho_{\mathfrak{o}}(g)(x) < x,  \\ 
g >_{\mathfrak{o}_x}  1 \; \Rightarrow  \; \rho_{\mathfrak{o}}(g)(x) > x,  
\end{array} \right. $$ 
there is an integer $n(g) > 0$ such that if $n \geq n(g)$, then 
$$\left\{ \begin{array}{l} 
g <_{\mathfrak{o}_x} 1 \; \Rightarrow \; \rho_{\mathfrak{o}}(g)(t(g_n)) < t(g_n)   \mbox{ and therefore }  \; g <_{\mathfrak{o}_{g_n}} 1\\ 
g >_{\mathfrak{o}_x}  1 \; \Rightarrow  \;\rho_{\mathfrak{o}}(g)(t(g_n)) > t(g_n)  \mbox{ and therefore } \; g >_{\mathfrak{o}_{g_n}} 1
\end{array} \right . $$ 
Fix a finite generating set of $G$ and let $B_j $ be the ball of radius $j$ in $G$ in the associated word metric. If $n_0 = \max \{n(g) \; | \; g \in B_j  \setminus \mbox{Stab}_{\rho_{\mathfrak{o}}}(x)\}$, then for $n \geq n_0$ we have 
$$B_j \cap (P(\mathfrak{o}_x)  \setminus \hbox{Stab}_{\rho_{\mathfrak{o}}}(x)) = B_j \cap (P(\mathfrak{o}_{t(g_n)}) \setminus \hbox{Stab}_{\rho_{\mathfrak{o}}}(x)),$$
 which proves the lemma. 
\end{proof}

Part (2) of Proposition \ref{prop: relation to orbit} is contained in the next lemma. 

\begin{lemma} 
\label{lemma: orders associated to ideal points 2}
Suppose that $x$ is an ideal point of $\mathfrak{o}$. Then $\mathfrak{O}_x \cap \overline{\mathcal{O}(\mathfrak{o})} \ne \emptyset$.
\end{lemma}

\begin{proof}
Let $\{g_n\}$ be a sequence of elements of $G$ such that $\lim_n t(g_n) = x$. Since $LO(G)$ is a compact metric space we can choose a subsequence $\{g_{n_j}\}$ such that $\{\mathfrak{o}_{t(g_{n_j})} \}$ converges  to a left-order $\hat{\mathfrak{o}} \in LO(G)$. Then $\lim_j P(\mathfrak{o}_{t(g_{n_j})}) = P(\hat{\mathfrak{o}})$ and therefore Lemma \ref{l: orders associated to ideal points 1} implies that 
$$P(\hat{\mathfrak{o}}) \setminus \hbox{Stab}_{\rho_{\mathfrak{o}}}(x) = P(\mathfrak{o}_x) \setminus \hbox{Stab}_{\rho_{\mathfrak{o}}}(x)$$
It follows that $\hat{\mathfrak{o}}$ and $\mathfrak{o}_x$ differ by a convex swap on $\hbox{Stab}_{\rho_{\mathfrak{o}}}(x)$, so $\hat{\mathfrak{o}} \in \mathfrak{O}_x$. On the other hand, Lemma \ref{lemma: conjugate orders 1}(1) shows that $\hat{\mathfrak{o}} = \lim_j \mathfrak{o}_{t(g_{n_j})} = \lim_j  g_{n_j} \cdot  \mathfrak{o}$, so $\hat{\mathfrak{o}}  \in \overline{\mathcal{O}(\mathfrak{o})}$. 
\end{proof}

\section{$3$-manifolds, their groups, and left-orderability}
\label{sec: lo and 3mflds}

\subsection{Prime and JSJ decompositions} Recall that every compact, connected, orientable $3$-manifold $M$ other than $S^3$ can be expressed as a connected sum of prime manifolds
\[ M \cong M_1 \sharp \dots \sharp M_n,
\]
and correspondingly, $\pi_1(M) = \pi_1(M_1) * \dots * \pi_1(M_n)$. Assuming that $\partial M$ contains no $2$-sphere components, $\pi_1(M_i) \ne \{1\}$ for each $i$ and therefore the fundamental group of $M$ is left-orderable if and only if each $\pi_1(M_i)$ is left-orderable (\cite{V}).  Observe that if $M_i$ is reducible for some $i$, then $\pi_1(M_i) \cong \mathbb{Z}$ and so $\pi_1(M_i)$ is left-orderable in this case; therefore the question of left-ordering $\pi_1(M)$ for an arbitrary $3$-manifold $M$ reduces to considering only the fundamental groups of irreducible $3$-manifolds.  In this case, a key tool is the following result.

\begin{theorem}\cite[Theorem 3.2]{BRW}
\label{thm: brw}
Suppose that $M$ is a compact, connected, orientable, irreducible $3$-manifold.  Then $\pi_1(M)$ is left-orderable if and only if there exists a surjection $\pi_1(M) \rightarrow L$ onto a left-orderable group $L$.
\end{theorem}

Note that as a corollary, if $M$ is a compact, connected orientable irreducible $3$-manifold with $\partial M$ a union of tori, an Euler characteristic argument shows that $b_1(M)>0$, so there exists a surjection $\pi_1(M) \rightarrow \mathbb{Z}$ and thus $\pi_1(M)$ is left-orderable \cite[Lemma 3.5]{BRW}.

We use this fact as follows.  If $M$ is orientable, irreducible and closed, the JSJ decomposition of $M$ provides a unique (up to isotopy) minimal collection $\mathcal{T}$ of embedded, disjoint incompressible tori such that $M\setminus \mathcal{T}$ consists of pieces $M_1, \dots, M_n$ where each $M_i$ is either Seifert fibered or atoroidal.  This decomposition allows one to realise $\pi_1(M)$ as the fundamental group of a graph of groups whose vertex groups are $\pi_1(M_1), \dots, \pi_1(M_n)$ and whose edge groups are $\pi_1(T) \cong \mathbb{Z} \oplus \mathbb{Z}$, where $T$ ranges over all tori in the collection $\mathcal{T}$.

By our observations above, if $M$ contains essential tori, then $\pi_1(M)$ is thus expressible as a fundamental group of a graph of groups, all of whose edge groups and vertex groups are left-orderable.  As no obstruction to left-orderability arises from considering these groups independently, the key to understanding left-orderability of $\pi_1(M)$ therefore lies in an analysis of the gluing maps used to reassemble $M$ from the pieces $M_i$, and the behaviour of the left-orderings of each $\pi_1(M_i)$ restricted to the components of $\partial M_i$ with respect to these gluing maps.  Such an analysis has already been completed in the case where all $M_i$ are Seifert fibred \cite{BC}. The analysis of the general case occupies much of the remainder of the manuscript.

\subsection{Slopes} 
A {\it slope} on the boundary of a knot manifold $M$ (cf. \S \ref{sec: intro} ) is an element $[\alpha]$ of the projective space of $H_1(\partial M; \mathbb R)$, where $\alpha \in H_1(\partial M; \mathbb R) \setminus \{0\}$. We use $\mathcal{S}(M)$ to denote the set of slopes on $\partial M$ topologised in the usual way, so that $\mathcal{S}(M) \cong S^1$. The subset $\mathcal{S}_{rat}(M)$ of {\it rational slopes} consists of those slopes represented by non-zero elements of $H_1(\partial M)$. Rational slopes can be identified as either 
\begin{itemize}

\item a $\pm$-pair of primitive elements of $H_1(\partial M) \equiv \pi_1(\partial M)$;

\vspace{.2cm} \item a $\partial M$-isotopy class of essential simple closed curves on $\partial M$.

\end{itemize} 
To each rational slope $[\alpha]$ on $\partial M$ or, more generally, on a torus boundary component $T$ of a $3$-manifold $W$, we can associate the $\alpha$-{\it Dehn filling} of $W$ given by $W(\alpha) = W \cup_f (S^1 \times D^2)$, where $f:\partial (S^1 \times D^2) \to T$ is a homeomorphism for which $f(\{*\} \times \partial D^2)$ is a simple closed curve of slope $[\alpha]$. A standard argument shows that $W(\alpha)$ is independent of the choice of $f$ up to a homeomorphism which is the identity on the complement in $W$ of a collar neighbourhood of $T$.

\section{Weak order-detection of slopes}

For the remainder of the paper we take $M$ to be a knot manifold unless otherwise indicated. That is, $M$ is a compact, connected, orientable, irreducible, boundary incompressible $3$-manifold such that $\partial M$ is a torus. As the first Betti number of a $M$ is at least $1$, Theorem \ref{thm: brw}  implies that $\pi_1(M)$ is left-orderable. 

We assume that the base points $x_0$ of our fundamental groups lie on $\partial M$, though we shall suppress them from our notation. For simplicity we denote the space of left-orders $LO(\pi_1(M))$  by $LO(M)$.

\label{sec: weak}

\subsection{Weak order-detection} 
\label{subsec: wk-ord-det}

If we think of $H_1(\partial M)$ as a lattice in $H_1(\partial M; \mathbb R) = H_1(\partial M) \otimes \mathbb R$, it is easy to see that every left-ordering $\mathfrak{o}$ of $\pi_1(M)$ determines a line in $H_1(\partial M; \mathbb R)$: The basic properties of positive cones imply that there is a line $L({\mathfrak{o}}) \subset H_1(\partial M; \mathbb R)$ uniquely determined by the fact that the elements of $H_1(\partial M) = \pi_1(\partial M)$ which lie to one side of it are $\mathfrak{o}$-positive and those lying to the other side are $\mathfrak{o}$-negative  (cf. \S \ref{subsec: lo z2}). We denote the slope of this line by $[L(\mathfrak{o})] \in \mathcal{S}(M)$, and it follows from a straightforward check that the slope function 
$$s: LO(M) \to \mathcal{S}(M) \cong S^1, s(\mathfrak{o}) = [L(\mathfrak{o})]$$        
is continuous \cite[Chapter 6]{Clay}.

\begin{definition}
{\rm A slope $[\alpha]$ on $\partial M$ is {\it weakly order-detected} if there is some $\mathfrak{o} \in LO(M)$ such that $[L(\mathfrak{o})] = [\alpha]$. } 
\end{definition}

Set 
$$\mathcal{D}_{ord}^{wk}(M) = \{[\alpha] \in \mathcal{S}(M) \; | \; [\alpha] \hbox{ is weakly order-detected}\}$$
and observe that as $\mathcal{D}_{ord}^{wk}(M)$ is the image of the compact space $LO(M)$ under the continuous map $s$, it is closed in $\mathcal{S}(M)$. 

\subsection{Weak order-detection and convexity}
It follows from the discussion in \S \ref{subsec: lo z2} that a rational slope $[\alpha]$ is weakly order-detected if and only if there is a left-order $\mathfrak{o} \in LO(M)$ such that $[L(\mathfrak{o})] \cap \pi_1(\partial M)$ is $\mathfrak{o}|_{\pi_1(\partial M)}$-convex. Our next result shows that we can define the weak order-detection of a rational slope in terms of relatively convex subgroups of $\pi_1(M)$. 

\begin{proposition}
\label{prop: wk ord-det means convex}
If a rational slope $[\alpha]$ is weakly order-detected, then there is a relatively convex subgroup $C$ of $\pi_1(M)$ such that $C \cap \pi_1(\partial M) = \langle \alpha \rangle$. Indeed, if $[\alpha] = [L(\mathfrak{o})]$ for some $\mathfrak{o} \in LO(M)$, there is a left-order $\mathfrak{o}' \in \overline{\mathcal{O}(\mathfrak{o})}$ and an $\mathfrak{o}' $-convex subgroup $C$ of $\pi_1(M)$ such that $C \cap \pi_1(\partial M) = \langle \alpha \rangle$.
\end{proposition}

\begin{proof}
Suppose that $[\alpha]$ is weakly order-detected by $\mathfrak{o} \in LO(M)$ and let $t: \pi_1(M) \to \mathbb R$ be an $\mathfrak{o}$-tight order-preserving embedding and $\rho_{\mathfrak{o}}: \pi_1(M) \to \hbox{Homeo}_+(\mathbb R)$ the associated dynamical realisation of $\mathfrak{o}$, as constructed in \S \ref{subsec: dyn real}. Since $[\alpha]$ is $\mathfrak{o}$-detected, $\langle \alpha \rangle$ is $\mathfrak{o}|_{\pi_1(\partial M)}$-convex. Hence it is $\mathfrak{o}|_{\pi_1(\partial M)}$-bounded (above and below) in $\pi_1(\partial M)$ (cf. \S \ref{subsec: convex subgroups}), and therefore $\mathfrak{o}$-bounded in $\pi_1(M)$. Thus the closure of the convex hull of $t(\langle \alpha \rangle)$ is a closed interval $[x_0, x_1]$, where $x_0 < 0 < x_1$. It is clear that $x_1 \in \hbox{Fix}_{\rho_{\mathfrak{o}}}(\alpha)$, and so it follows from Lemma \ref{l: stabilisers} that $x_1$ is an ideal point of $\mathfrak{o}$. Thus, $[x_0, x_1] \cap t(\pi_1(\partial M)) = t(\langle \alpha \rangle)$. If $\beta \in \pi_1(\partial M) \setminus \langle \alpha \rangle$, the $\mathfrak{o}|_{\pi_1(\partial M)}$-convexity of $\langle \alpha \rangle$ implies that for all integers $n, m$ 
\begin{itemize}

\item $\beta \alpha^n <_\mathfrak{o} \alpha^m$ when $\beta <_\mathfrak{o} 1$, and 

\vspace{.2cm} \item  $\alpha^m <_\mathfrak{o} \beta \alpha^n$ when $1 <_\mathfrak{o} \beta$.

\end{itemize}
Thus either $\rho(\beta)([x_0, x_1]) \subset (-\infty, x_0]$ or $\rho(\beta)([x_0, x_1]) \subset [x_1, \infty)$. In either case $\rho(\beta)(x_1) = x_1$ is impossible, so $\hbox{Stab}_{\rho_{\mathfrak{o}}}(x_1) \cap \pi_1(\partial M) = \langle \alpha \rangle$. 

Recall the family $\mathfrak{O}_{x_1}$ of left-orders associated to $\mathfrak{o}$ and its ideal point $x_1$, as defined in \S \ref{subsec: left-orders from dyn reals}. By construction, $C = \hbox{Stab}_{\rho_{\mathfrak{o}}}(x_1)$ is convex with respect to each left-order in $\mathfrak{O}_{x_1}$. Further, by Lemma \ref{lemma: orders associated to ideal points 2} there is a left-order $\mathfrak{o}' \in \overline{\mathcal{O}(\mathfrak{o})} \cap \mathfrak{O}_{x_1}$. Then $C$ is $\mathfrak{o}'$-convex and $C \cap \pi_1(\partial M) = \langle \alpha \rangle$, which completes the proof. 
\end{proof}

The set of weakly order-detected slopes can be proper in $\mathcal{S}(M)$, for instance this is often the case when $M$ is Seifert fibred, but in this case $\pi_1(\partial M)$ is never contained in a relatively convex proper subgroup of $\pi_1(M)$, as the following proposition will show. 

\begin{proposition} 
\label{prop: not bdry-cofinal implies all slopes}
Suppose that $\mathfrak{o}$ is a left-order on $\pi_1(M)$ for which $\pi_1(\partial M)$ is contained in a proper $\mathfrak{o}$-convex subgroup $C$ of $\pi_1(M)$. Then $\mathcal{D}_{ord}^{wk}(M) = \mathcal{S}(M)$. More precisely, there is a family $\mathcal{F}$ of left-orders on $\pi_1(M)$ whose positive cones differ only on $C$ such that for each slope $\alpha$ on $\partial M$ there is a left-order $\mathfrak{o}_\alpha \in \mathcal{F}$ on $\pi_1(M)$ which weakly detects $[\alpha]$. 
\end{proposition}

\begin{proof}
Since $C$ has infinite index in $\pi_1(M)$, the cover $W \to M$ such that $\pi_1(W) = C$ is non-compact. By construction, $\partial M$ lifts to a torus $T \subseteq \partial W$. We divide the proof into a series of claims. 

{\bf Claim 1.}  {\it The set $Z$ of rational slopes $\beta$ on $\partial M = T$ such that the image of the homomorphism $H_1(T) \to H_1(W(\beta))$ is zero is a discrete subset of $\mathcal{S}(M)$.}

\begin{proof}[Proof of Claim 1]
Consider the inclusion $i: T \to W$ and suppose that the homomorphism $H_1(T; \mathbb Q) \to H_1(W; \mathbb Q)$ it induces is zero. Then there is a compact $W_0 \subset W$ containing $T$ such that $H_1(T; \mathbb Q) \to H_1(W_0; \mathbb Q)$ is zero. But in this case we can attach compact $3$-manifolds to $W_0$ along $\partial W_0 \setminus T$ to create a compact, orientable $3$-manifold $W_1$ with boundary $T$ for which $H_1(T; \mathbb Q) \to H_1(W_1; \mathbb Q)$ is zero, which is impossible (\cite[Lemma 3.5]{Hat}). Thus the image of $H_1(T; \mathbb Q) \to H_1(W; \mathbb Q)$ is at least $\mathbb Q$.

If the image of $H_1(T; \mathbb Q) \to H_1(W; \mathbb Q)$ is $\mathbb Q^2$, then for each primitive $\beta \in H_1(T)$, the image of $H_1(T; \mathbb Q)$ in $H_1(W(\beta); \mathbb Q)$ has dimension $1$, so that $Z  = \emptyset$ and the claim holds. 

If the image of $H_1(T; \mathbb Q) \to H_1(W; \mathbb Q)$ is $\mathbb Q$, the kernel of $H_1(T) \to H_1(W; \mathbb Q)$ is a summand of rank $1$. Fix a primitive element $\gamma_0 \in H_1(T)$ spanning this kernel and let $\gamma_1 \in H_1(T)$ be a primitive element dual to $\gamma_0$. Then if $i_*: H_1(T) \to H_1(W)$ is the inclusion-induced homomorphism, $i_*(\gamma_1)$ has infinite order in $H_1(W)$. It follows that for $\beta$ a primitive element of $H_1(T)$, the order of $\gamma_1$ in $H_1(W(\beta))$ is a non-zero multiple of $\Delta(\beta, \gamma_0) = |\beta \cdot \gamma_0|$. In particular, $H_1(W(\beta))$ is non-trivial if $\Delta(\beta, \gamma_0) \ne 1$. It follows that $Z$ is contained in $\{[\gamma_1 + n \gamma_0] \; | \; n \in \mathbb Z\}$, which is a discrete subset of $\mathcal{S}(M)$.  
\end{proof} 

Let $Z^*$ be the union of $Z$ and the set of rational slopes $[\alpha]$ on $T$ such that $W(\alpha)$ is reducible. Gordon and Luecke have shown that there are at most three rational slopes of the latter type (\cite{GLu}), so $Z^*$ is a nowhere dense subset of $\mathcal{S}(M)$. Hence $\mathcal{S}_{rat}(T) \setminus Z^*$ is dense in $\mathcal{S}(M)$.

 {\bf Claim 2.} {\it If $[\alpha]$ is a rational slope on $T$ not contained in $Z^*$, then $\pi_1(W(\alpha))$ is infinite and torsion-free.} 

\begin{proof}[Proof of Claim 2]
Since $[\alpha] \not \in Z^*$, $W(\alpha)$ is irreducible. Further, Claim 1 shows that $\pi_1(W(\alpha)) \ne 1$, and so as an irreducible, non-compact $3$-manifold has a torsion-free fundamental group, the conclusion of Claim 2 holds (\cite[(C.1) and (C.2), page 37]{AFW}). 
\end{proof} 

{\bf Claim 3.} {\it If $[\alpha]$ is a rational slope on $T$ not contained in $Z^*$, then the image of $\pi_1(T)$ in $\pi_1(W(\alpha))$ is infinite cyclic.} 

\begin{proof}[Proof of Claim 3]
The  inclusion-induced homomorphism $\pi_1(T) \to \pi_1(W(\alpha))$ factors through $\pi_1(T)/\langle \alpha \rangle \cong \mathbb Z$, so its image is a cyclic group. On the other hand, we saw in the proof of Claim 1 that the image $H_1(T) \to H_1(W(\alpha))$ is non-zero when $[\alpha] \not \in Z$, so Claim 2 shows that the image of $\pi_1(T) $ in $\pi_1(W(\alpha))$ is infinite cyclic. 
\end{proof} 

{\bf Claim 4.} {\it If $[\alpha]$ is a rational slope on $T$ not contained in $Z^*$, then $\pi_1(W(\alpha))$ is left-orderable.}

\begin{proof}[Proof of Claim 4] 
We will show that $\pi_1(W(\alpha))$ is locally indicable, hence left-orderable by the Burns-Hale theorem \cite{BH}. To that end let $H$ be a non-trivial finitely generated subgroup of $\pi_1(W(\alpha))$ and $V \to W(\alpha)$  the associated cover. 

If $V \to W(\alpha)$ is a finite cover, then $\pi_1(W(\alpha))$ is finitely generated. We claim that at least one boundary component of a compact core $N$ of the non-compact manifold $W(\alpha)$ has genus $1$ or more. Otherwise each is a $2$-sphere, so the irreducibility of $W(\alpha)$ couples with the non-triviality of its fundamental group to show that each boundary component of $N$ bounds a $3$-ball contained in $W(\alpha) \setminus \mbox{int}(N)$, which is impossible as it would imply that $W(\alpha)$ is compact. It follows that the first Betti number of $N$, and hence $W(\alpha)$, is positive (\cite[Lemma 3.5]{BRW}), so the same is true of any of its finite degree covers. In particular there is a surjection $H = \pi_1(V) \to \mathbb Z$. 

On the other hand, if $V \to W(\alpha)$ is an infinite degree cover, an argument of Howie and Short shows that the first Betti number of $V$ is positive (cf. the proof of \cite[ Theorem 1.1]{BRW}). Thus $\pi_1(W(\alpha))$ is locally indicable. 
\end{proof} 
Now we can complete the proof of the proposition. 

Given $\alpha \in \mathcal{S}_{rat}(T) \setminus Z^*$, the exact sequence $1 \to \langle \langle \alpha \rangle \rangle_C \to C \to C/\langle \langle \alpha \rangle \rangle_C = \pi_1(W(\alpha)) \to 1$ determines a left-order $\mathfrak{o}_C$ on $C$ for which $\langle \langle \alpha \rangle \rangle_C$ is a proper convex subgroup. Claim 3 shows that  $\langle \langle \alpha \rangle \rangle_C \cap \pi_1(T) = \langle \alpha \rangle$, and therefore $\langle \alpha \rangle$ is a proper $\mathfrak{o}_C|_{\pi_1(\partial M)}$-convex subgroup of $\pi_1(\partial M) \leq C$. Proposition \ref{prop: still convex} then shows that $P(\mathfrak{o}_C) \sqcup \big(P(\mathfrak{o}) \setminus C\big)$ is the positive cone of a left-order on $\pi_1(M)$ which weakly order-detects $[\alpha]$. 
It follows that $\mathcal{S}_{rat}(T) \setminus Z^*$ is contained in the image of the map $LO(M) \rightarrow \mathcal{S}(M), \mathfrak{o} \mapsto [L(\mathfrak{o})]$.  Since this map is continuous and $LO(M)$ is compact, density of $\mathcal{S}_{rat}(T) \setminus Z^*$ in $\mathcal{S}(M)$ yields the result.
\end{proof} 

Our next two corollaries show how one can use epimorphisms and bi-orders to satisfy the convexity condition required in Proposition \ref{prop: not bdry-cofinal implies all slopes}.

\begin{corollary} \label{cor: all weak}
Suppose that $\varphi: \pi_1(M) \to G$ is an epimorphism. Then $\mathcal{D}_{ord}^{wk}(M) = \mathcal{S}(M)$ if either of the following two conditions holds: 

$(1)$ $\pi_1(\partial M) \leq \hbox{kernel}(\varphi)$ and $G$ is left-orderable; 

$(2)$ $\pi_1(\partial M) \cap \hbox{kernel}(\varphi) = \{1\}$ and $G$ is bi-orderable.

\end{corollary}

\begin{proof}
In each case it suffices to show that $\pi_1(\partial M)$ is contained in a relatively convex proper subgroup of $\pi_1(M)$ by Proposition \ref{prop: not bdry-cofinal implies all slopes}. 

In the case of assertion (1) of the corollary, the left-orderability of $\pi_1(M)$ and that of $G$ combines with Proposition \ref{prop: action yields lo} to show that $\hbox{kernel}(\varphi)$ is relatively convex, so the claim follows. 

For (2), first note that if $G$ is abelian then every bi-order of $\varphi(\pi_1(\partial M)) \cong \mathbb{Z} \oplus \mathbb{Z}$ extends to a bi-order of $G$ \cite[Theorem 4]{Rhe}, so every slope is weakly detected by applying Proposition \ref{prop: action yields lo}. 

On the other hand, if $G$ is nonabelian then fix a left-order on $\pi_1(M)$ and let $\mathfrak{o}$ be a bi-order on $G$. If $Z_g$ denotes the centraliser of any element $g \in G$, then $\varphi^{-1}(Z_g)$ is the stabiliser in $\pi_1(M)$ of $g$ under the $\mathfrak{o}$-preserving action $g \mapsto \varphi(\gamma)g\varphi(\gamma)^{-1}$ of $\pi_1(M)$ on $G$. Since $\pi_1(M)$ is left-orderable so is the subgroup $\varphi^{-1}(Z_g)$, and therefore Proposition \ref{prop: action yields lo} produces a left-order $\mathfrak{o}_g$ on $\pi_1(M)$ for which $\varphi^{-1}(Z_g)$ is convex. 

Using this construction we proceed as follows.  Fix non-identity elements $\alpha_0, \alpha_ 1 \in \pi_1(\partial M)$ that generate $\pi_1(\partial M)$, and suppose that $Z_{\varphi(\alpha_i)} \ne G$ for some $i \in \{0, 1\}$. Then $\varphi^{-1}( Z_{\varphi(\alpha_i)}) \ne \pi_1(M)$ is a proper $\mathfrak{o}_{\alpha_i}$-convex subgroup of  $\pi_1(M)$ which contains $\pi_1(\partial M)$.   On the other hand, if $Z_{\varphi(\alpha_i)} = G$ for $i = 0,1$ then $\varphi( \pi_1(\partial M))$ is contained in the centre $Z(G)$ of $G$, which is a proper subgroup since $G$ is nonabelian.  Now since $G$ is bi-orderable, $Z(G)$ is relatively convex subgroup of $G$ by \cite[Theorem 2.2.4]{MR} and thus $\varphi^{-1}(Z(G))$ is a proper, relatively convex subgroup of $\pi_1(M)$ containing $\pi_1(\partial M)$.
\end{proof}

Recall that if $M$ is a knot manifold and $H_1(M; \mathbb{Q}) = \mathbb{Q}$, then a rational longitude of $M$ is a primitive element $\lambda_M \in H_1(\partial M)$ such that $\lambda_M$ has finite order in $H_1(M)$.

\begin{corollary}
Suppose that $H_1(M; \mathbb{Q}) = \mathbb{Q}$, $G$ is a bi-orderable group and $\varphi: \pi_1(M) \to G$ is an epimorphism, and that $\varphi(\lambda_M) \neq 1$.  Then $\mathcal{D}_{ord}^{wk}(M) = \mathcal{S}(M)$.
\end{corollary}
\begin{proof}
Let $\mu \in \pi_1(\partial M)$ be an element dual to $\lambda_M$, and note that the image of $\mu$ under any non-trivial homomorphism from $\pi_1(M)$ to a torsion-free abelian group must be nontrivial, while the image of $\lambda_M$ must be trivial.  

Now since $M$ is compact, $\pi_1(M)$ is finitely generated and therefore so is $G$; so by \cite[Theorem 2.19]{CR} there exists an epimorphism $G \rightarrow \mathbb{Z}$.  The image of $\mu$ under the composition $\pi_1(M) \to G \rightarrow \mathbb{Z}$ is nontrivial, and so $\varphi(\mu)^m \neq \varphi(\lambda_M)^n$ unless $m=n=0$.  The result now follows from Corollary \ref{cor: all weak}(2).
\end{proof}

\begin{example} 
\label{example: adam's eg}
{\rm Suppose that $\pi_1(M)$ admits a bi-order $\mathfrak{o}$, for instance $M$ could be the figure eight knot exterior. Then taking $\varphi$ to be the identity in Corollary \ref{cor: all weak}(2) implies that every slope on $\partial M$ is weakly order-detected.}
\end{example}

\subsection{Weak order-detection and boundary-cofinality} 

In this section we introduce the concept of boundary cofinality, which turns out to be essential in the study of slope detection.  In particular, it is closely related to whether or not $\mathcal{S}(M) = \mathcal{D}_{ord}^{wk}(M)$ (Theorem \ref{thm: some not weak implies all cofinal}) and is key in connecting the notions of weak order-detection and order-detection (proof of Theorem \ref{thm: some not weak means weak detd = detd}).

\begin{definition}
{\rm A left-order $\mathfrak{o} \in LO(M)$ is {\it boundary-cofinal} if $\pi_1(\partial M)$ is $\mathfrak{o}$-cofinal.}
\end{definition}

\begin{example} 
\label{example: torus knot cofinality}
It was shown in \cite[Proposition 4.7(2)]{BC} that if $M$ is a torus knot exterior, then each $\mathfrak{o} \in LO(M)$ is a boundary-cofinal. 
\end{example}

Here we convert the convexity condition on $\pi_1(\partial M)$ of Proposition \ref{prop: not bdry-cofinal implies all slopes} into a boundary-cofinality condition. 

\begin{lemma} 
\label{cofinal versus convex}
Let  $\mathfrak{o}$ be a left-order on $\pi_1(M)$. 

$(1)$ If $\pi_1(\partial M)$ is contained in a proper $\mathfrak{o}$-convex subgroup of $\pi_1(M)$, then $\mathfrak{o}$ is not boundary-cofinal.  

$(2)$ If $\mathfrak{o}$ is not boundary-cofinal, there is a left-order $\mathfrak{o}' \in \overline{\mathcal{O}(\mathfrak{o})}$ and an $\mathfrak{o}'$-convex subgroup of $\pi_1(M)$ which contains $\pi_1(\partial M)$.
\end{lemma}

\begin{proof}
(1) Say that $\pi_1(\partial M)$ is contained in a proper $\mathfrak{o}$-convex subgroup $C$ of $\pi_1(M)$. Since $C$ is proper, there is some $\gamma \in (\pi_1(M) \setminus C) \cap P(\mathfrak{o})$, and since $C$ is convex, $\gamma^{-1} < c < \gamma$ for all $c \in C$. In particular, this holds for all $c \in \pi_1(\partial M)$. Hence $\pi_1(\partial M)$ is not $\mathfrak{o}$-cofinal.

(2) Suppose that $\pi_1(\partial M)$ is not $\mathfrak{o}$-cofinal and let $\alpha \in \pi_1(\partial M)$ be $\mathfrak{o}|_{\pi_1(\partial M)}$-cofinal (cf. \S \ref{subsec: lo z2}). Let $t: \pi_1(M) \to \mathbb R$ be a tight order-preserving embedding, as used in the construction of a dynamic realisation $\rho_\mathfrak{o}$ of $\mathfrak{o}$. Since $\langle \alpha \rangle$ is not $\mathfrak{o}$-cofinal, it is either $\mathfrak{o}$-bounded above or $\mathfrak{o}$-bounded below by some element of $\pi_1(M)$. Hence, after possibly replacing $\alpha$ by its inverse, we can assume that $\lim_n t(\alpha^n) = x \in \mathbb R$. Then $\rho_\mathfrak{o}(\alpha)(x) = \lim_n t(\alpha \cdot \alpha^n) = \lim_n t(\alpha^{n+1}) = x$. If $\beta \in \pi_1(\partial M)$ is also not $\mathfrak{o}|_{\pi_1(\partial M)}$-cofinal and has the same sign as $\alpha$, then it is easy to see that $\lim_n t(\beta^n) = x$ and therefore $\rho_\mathfrak{o}(\beta)(x) = x$. Since we can always choose $\alpha, \beta$ to form a generating set of $\pi_1(\partial M)$, it follows that $\pi_1(\partial M) \leq \hbox{Stab}_{\rho_\mathfrak{o}}(x)$. On the other hand, $\hbox{Stab}_{\rho_\mathfrak{o}}(x)$ is convex with respect to any left-order in $\mathfrak{O}_x$ (cf. Definition \ref{def: O_x}). Lemma \ref{lemma: orders associated to ideal points 2} finishes the proof. 
\end{proof}

{\bf Theorem \ref{thm: some not weak implies all cofinal}.} 
{\it If there is a slope on $\partial M$ which is not weakly order-detected, then each $\mathfrak{o} \in LO(M)$ is boundary-cofinal. }

\begin{proof}
Fix $\mathfrak{o} \in LO(M)$ and suppose that $\pi_1(\partial M)$ is not $\mathfrak{o}$-cofinal. Then Lemma \ref{cofinal versus convex}(2) implies that there is a left-order $\mathfrak{o}' \in LO(M)$ and a proper $\mathfrak{o}'$-convex subgroup of $\pi_1(M)$ which contains $\pi_1(\partial M)$. Proposition \ref{prop: not bdry-cofinal implies all slopes} then shows that $\mathcal{D}_{ord}^{wk}(M) = \mathcal{S}(M)$, contradicting our hypotheses. Thus, $\pi_1(\partial M)$ is $\mathfrak{o}$-cofinal for each $\mathfrak{o} \in LO(M)$. 
\end{proof}

\section{A dynamical interpretation of boundary-cofinality}
\label{sec: bdry cof and dynamics}

In this section we interpret boundary-cofinality in terms of actions on the reals to connect it with work of Nie (\cite{Nie1, Nie2}).

\begin{lemma}
\label{lemma: bdry-cofinal implies fpf} 
Let $\mathfrak{o} \in LO(M)$ and $\rho_\mathfrak{o}: \pi_1(M) \to \mbox{Homeo}_+(\mathbb R)$ a dynamic realisation of $\mathfrak{o}$.  

$(1)$ $\gamma \in \pi_1(M)$ is $\mathfrak{o}$-cofinal if and only if $\rho_\mathfrak{o}(\gamma)$ acts fixed-point freely on $\mathbb R$. 

$(2)$ $\mathfrak{o}$ is boundary-cofinal if and only if each element of $\rho_\mathfrak{o}(\pi_1(\partial M) \setminus [L(\mathfrak{o})])$ acts fixed-point freely on $\mathbb R$.

$(3)$ For each $\alpha \in [L(\mathfrak{o})] \cap \pi_1(\partial M)$, $\rho_\mathfrak{o}(\alpha)$ has fixed points in $\mathbb R$. 

\end{lemma}

\begin{proof}
Suppose that $\mathfrak{o}$ is boundary-cofinal and let $t: \pi_1(M) \to \mathbb R$ be the tight order-preserving embedding used to construct $\rho_\mathfrak{o}$. 
If  $\gamma \in \pi_1(M)$ is $\mathfrak{o}$-cofinal then $t(\langle \gamma \rangle)$ is unbounded above and below in $\mathbb R$, from which it follows that $\rho_\mathfrak{o}(\gamma)$ acts fixed-point freely on $\mathbb R$. Conversely, if $\rho_\mathfrak{o}(\gamma)$ acts fixed-point freely on $\mathbb R$, note that as the supremum and infinum of $t(\langle \gamma \rangle)$ (thought of as lying in $[-\infty, +\infty]$) are fixed by $\rho_\mathfrak{o}(\gamma)$, $t(\langle \gamma \rangle)$ must be unbounded above and below. Thus $\gamma$ is $\mathfrak{o}$-cofinal, which proves (1). 

For (2), recall from \S \ref{subsec: lo z2} that each element $\alpha$ of $\pi_1(\partial M) \setminus [L(\mathfrak{o})]$ is $\mathfrak{o}|_{\pi_1(\partial M)}$-cofinal, so the boundary-cofinality of $\mathfrak{o}$ implies that each such $\alpha$ is $\mathfrak{o}$-cofinal. Thus (1) implies (2). 

Assertion (3) clearly holds for $\alpha = 1$, so suppose otherwise. If there exists $\alpha \neq 1$ in $[L(\mathfrak{o})] \cap \pi_1(\partial M)$, then $[L(\mathfrak{o})] \cap \pi_1(\partial M) \cong \mathbb Z$ and without loss of generality we can suppose that $[L(\mathfrak{o})] \cap \pi_1(\partial M) = \langle \alpha \rangle$. Then $\langle \alpha \rangle$ is an $\mathfrak{o}|_{\pi_1(\partial M)}$-convex. Hence it is $\mathfrak{o}$-bounded above and below in $\pi_1(\partial M)$. It follows that $\rho_\mathfrak{o}(\alpha)$ fixes $x = \sup t(\langle \alpha \rangle)$, which completes the proof of (2). 
\end{proof}

This lemma leads to a reformulation of the notion of boundary-cofinality.  

\begin{proposition}
\label{prop: bdry cofinal iff no fixed pts}
The following statements are equivalent.

$(1)$ Each $\mathfrak{o} \in LO(M)$ is boundary-cofinal.

$(2)$ If $\rho: \pi_1(M) \to \mbox{Homeo}_+(\mathbb R)$ has no fixed point in $\mathbb R$, then $\rho|_{\pi_1(\partial M)}$ has no fixed points in $\mathbb R$. 
 
\end{proposition}

\begin{proof} 
Assume (1), and suppose that  $\rho: \pi_1(M) \to \mbox{Homeo}_+(\mathbb R)$ is such that $\rho|_{\pi_1(\partial M)}$ has a fixed point $x_0 \in \mathbb{R}$.  Then as $\pi_1(M)$ is left-orderable, the subgroup $\mbox{Stab}_\rho(x_0)$ is also left-orderable, and it also contains $\pi_1(\partial M)$.  By Proposition \ref{prop: action yields lo}, there is a left-order $\mathfrak{o}_{x_0}$ on $\pi_1(M)$ relative to which $\mbox{Stab}_\rho(x_0)$ is convex. However as $\mathfrak{o}_{x_0}$ must be boundary-cofinal, its definition in Proposition \ref{prop: action yields lo} forces $\mbox{Stab}_\rho(x_0) = \pi_1(M)$. Thus (2) holds.

In the other direction, suppose (2) holds and let $\rho_{\mathfrak{o}}: \pi_1(M) \to \mbox{Homeo}_+(\mathbb R)$ be a dynamic realisation of $\mathfrak{o} \in LO(M)$ constructed using the tight embedding $t: \pi_1(M) \rightarrow \mathbb{R}$. By construction, $\rho_{\mathfrak{o}}$ has no global fixed point, so condition (2) implies that the same is true for $\rho_{\mathfrak{o}}|_{\pi_1(\partial M)}$. Then $\{ \rho_{\mathfrak{o}}(\alpha)(0)(t(\alpha)) \; | \; \alpha \in \pi_1(\partial M) \}$ must be unbounded above and below, otherwise the supremum or infimum of this set would be bounded and be fixed by the action of every element in $\pi_1(\partial M)$. Thus $\mathfrak{o} \in LO(M)$ is boundary-cofinal.
\end{proof}

Statement (2) of the proposition appears as condition (a) on a knot manifold $M$ in recent work of Nie (\cite{Nie1}). Nie had considered the exteriors of knots in $S^3$, but it is natural to work in full generality.  

\begin{corollary}
\label{cor: bdry cofinal iff no fixed pts}
Every left-order $\mathfrak{o} \in LO(M)$ is boundary-cofinal if and only if each homomomorphism $\rho: \pi_1(M) \to \mbox{Homeo}_+(\mathbb R)$ which has no fixed point in $\mathbb R$, has no fixed points when restricted to $\pi_1(\partial M)$. 
\qed 
\end{corollary}

Nie has shown that many $L$-space knots satisfy property (a), including those $(1,1)$-knots which are $L$-space knots (\cite{Nie1}). This family includes $1$-bridge braid knots and certain families of twisted torus knots. We expect property (a) to hold for all $L$-space knot groupss and, more generally, the fundamental groups of Floer simple manifolds (cf. Conjecture \ref{conj: floer simple implies bdry-cofinal}).

\section{Order-detection of slopes} 
\label{sec: reg}
Here we introduce a form of slope detection adapted to understanding the left-orderability of the fundamental groups of $3$-manifolds obtained gluing knot manifolds by together along their boundary tori. We begin with an analysis of how weak order-detection behaves with respect to the action of $\pi_1(M)$ on $LO(M)$.  

\subsection{Weak order-detection and peripheral subgroups}
A {\it peripheral subgroup} of $\pi_1(M)$ is a subgroup of the form $g \pi_1(\partial M) g^{-1}$, where $g \in \pi_1(M)$.  If $\mathfrak{o} \in LO(M)$, the restriction $\mathfrak{o}|_{g \pi_1(\partial M) g^{-1}}$ determines a line $ L(\mathfrak{o}; g \pi_1(\partial M) g^{-1})$ in $H_1(g \pi_1(\partial M) g^{-1}; \mathbb R) = g \pi_1(\partial M) g^{-1} \otimes \mathbb R$ in the usual way: all elements of $g \pi_1(\partial M) g^{-1}= H_1(g \pi_1(\partial M) g^{-1})$ which lie to one side of it are $\mathfrak{o}$-positive and all elements lying to the other side are $\mathfrak{o}$-negative. This line determines a slope $[L({\mathfrak{o}}; g \pi_1(\partial M) g^{-1})]$ on $\partial M$ using the isomorphism
\vspace{-.2cm} 
\begin{center} 
\begin{tikzpicture}[scale=0.7]

\node at (.5, 4.5) {$H_1(\pi_1(\partial M); \mathbb R) = $};
\node at (4, 4.5) {$\pi_1(\partial M) \otimes \mathbb R$};
\node at (14, 4.5) {$(g \pi_1(\partial M) g^{-1}) \otimes \mathbb R$};
\node at (18.7, 4.5) {$= H_1(g \pi_1(\partial M) g^{-1}; \mathbb R)$};

\node at (8.5, 4.9) {$\beta \otimes t \mapsto (g^{-1} \beta g) \otimes t$}; 

\draw [ ->] (5.5, 4.5) --(11.8,4.5); 

\end{tikzpicture}
\end{center} 

The following example shows that in general, $[L(\mathfrak{o}; g \pi_1(\partial M) g^{-1})]$ can vary with the choice of peripheral subgroup.  

\begin{example} 
\label{example: variation with peripheral subgroup}
Suppose that $\pi_1(M)$ admits a bi-order $\mathfrak{o}$; for instance $M$ could be the figure eight knot exterior. Then taking $\varphi$ to be the identity in Corollary \ref{cor: all weak}(2) and considering its proof shows that if we take $\alpha \in \pi_1(\partial M)$ with centraliser $Z_\alpha = \pi_1(\partial M)$ (\cite[Theorem 1]{Sim}), then $\pi_1(M)$ admits a left-order $\mathfrak{o}_\alpha$ for which $\pi_1(\partial M)$ is convex. Hence, we can create new left-orders on $\pi_1(M)$ by  arbitrarily altering $\mathfrak{o}_\alpha$ on $\pi_1(\partial M)$ and leaving it unaltered on $\pi_1(M) \setminus \pi_1(\partial M)$ (Proposition \ref{prop: still convex}(2)). In particular, this can be done so that $[L(\mathfrak{o}; \pi_1(\partial M))] \ne [L(\mathfrak{o}; g\pi_1(\partial M)g^{-1})]$ for an appropriately chosen $g \in \pi_1(M)$. 
\end{example}

The next proposition shows how boundary-confinality guarantees that the slope $[L(\mathfrak{o}); g \pi_1(\partial M)g^{-1}]$ is independent of the choice of $g \in \pi_1(M)$. 

\begin{proposition}
\label{prop: bdry cofinal implies invt} 
Suppose that $\mathfrak{o} \in LO(M)$ is boundary-cofinal. Then 
$$[L(\mathfrak{o}; \pi_1(\partial M))] = [L(\mathfrak{o}; g\pi_1(\partial M)g^{-1})]$$ 
for all $g \in \pi_1(M)$. 

\end{proposition} 

\begin{proof} 
Fix $g \in G$ and let $\rho_\mathfrak{o}: \pi_1(M) \to \mbox{Homeo}_+(\mathbb R)$ be a dynamic realisation of $\mathfrak{o}$. It follows from \S \ref{subsec: lo z2} that each $\gamma \in \pi_1(\partial M) \setminus [L(\mathfrak{o})]$ is $\mathfrak{o}|_{\pi_1(\partial M)}$-cofinal, and since we have assumed that $\mathfrak{o}$ is boundary-cofinal, each such $\gamma$ is $\mathfrak{o}$-cofinal. Lemma \ref{lemma: bdry-cofinal implies fpf}(1) shows that $\rho_\mathfrak{o}(\gamma)$ acts fixed-point freely on $\mathbb R$, so the same is true of $\rho_\mathfrak{o}(g \gamma g^{-1})$. Another application of Lemma \ref{lemma: bdry-cofinal implies fpf}(1) then shows that $g \gamma g^{-1}$ is $\mathfrak{o}$-cofinal. Further, it follows from \cite[Lemma 4.6(1)]{BC} that $g \gamma g^{-1} >_\mathfrak{o} 1$ if and only if $\gamma >_\mathfrak{o} 1$. Thus, each element of $g (\pi_1(\partial M) \setminus [L(\mathfrak{o})])g^{-1} $ is $\mathfrak{o}$-cofinal and of the same $\mathfrak{o}$-sign as $\gamma$. It follows that 
$[L(\mathfrak{o}; g\pi_1(\partial M)g^{-1})] = [L(\mathfrak{o}; \pi_1(\partial M))]$.
\end{proof}

We end our discussion of weak order-detection with a lemma that will prove important in our gluing theorem. 

\begin{lemma}
\label{lemma: variance under action} 
For every order $\mathfrak{o} \in LO(M)$ and $g \in \pi_1(M)$, we have
$$[L(g^{-1} \cdot \mathfrak{o}; \pi_1(\partial M))] = [L(\mathfrak{o}; g \pi_1(\partial M) g^{-1})]$$ 
\end{lemma}

\begin{proof} 
Recall that for $\alpha, \beta \in \pi_1(\partial M)$ and $g \in \pi_1(M)$ we have 
$$\alpha <_{g^{-1} \cdot \mathfrak{o}} \beta \; \Leftrightarrow \; g \alpha g^{-1} <_{\mathfrak{o}} g \beta g^{-1}$$ 
It follows that there is a bijection
\begin{center} 
\begin{tikzpicture}[scale=0.8]
\node at (4, 4.5) {$P(g^{-1}  \cdot \mathfrak{o}) \cap \pi_1(\partial M)$};
\node at (14, 4.5) {$P(\mathfrak{o}) \cap g \pi_1(\partial M) g^{-1}$};

\node at (8.5, 4.9) {$\beta \mapsto g \beta g^{-1}$}; 

\draw [ ->] (5.9, 4.5) --(11.9,4.5); 

\end{tikzpicture}
\end{center} 
and therefore $L(g^{-1} \cdot \mathfrak{o}; \pi_1(\partial M))$ maps to $L(\mathfrak{o}; g \pi_1(\partial M) g^{-1})$ under the isomorphism
$$H_1(\pi_1(\partial M); \mathbb R) = \pi_1(\partial M) \otimes \mathbb R \xrightarrow{\;\;\; \beta \otimes t \; \mapsto \; (g^{-1} \beta g) \otimes t \;\;\;} 
(g \pi_1(\partial M) g^{-1}) \otimes \mathbb R = H_1(g \pi_1(\partial M) g^{-1}; \mathbb R)$$ 
Thus $[L(g^{-1} \cdot \mathfrak{o}; \pi_1(\partial M))] = [L(\mathfrak{o}; g \pi_1(\partial M) g^{-1})]$. 
\end{proof} 
Here is an immediate consequence of the last two results.

\begin{corollary}
\label{cor: invariance under action}
If $\mathfrak{o} \in LO(M)$ is boundary-cofinal, then $[L(g \cdot \mathfrak{o}; \pi_1(\partial M))] = [L(\mathfrak{o}; \pi_1(\partial M))]$ for each $g \in \pi_1(M)$. 
\qed
\end{corollary}

\subsection{Order-detection}

Set 
$$LO(M; [\alpha]) = \{\mathfrak{o} \in LO(M) \; | \; [L(g\cdot \mathfrak{o}; \pi_1(\partial M))]  = [\alpha] \hbox{ for all } g \in \pi_1(M)\}$$ 

We have noted elsewhere that the slope map $s: LO(M) \to \mathcal{S}(M), s(\mathfrak{o}) = [L(\mathfrak{o})]$, is continuous and that given $g \in \pi_1(M)$, the map $\mathfrak{o} \mapsto g \cdot \mathfrak{o}$ is a homeomorphism of $LO(M)$.  Hence the identity
$$LO(M; [\alpha]) = \bigcap_{g \in \pi_1(M)} g \cdot (s^{-1}([\alpha]))$$ 
expresses $LO(M; [\alpha])$ as an intersection of closed subsets of $LO(M)$, and it is therefore closed itself. 

\begin{definition}
{\rm A slope $[\alpha]$ on $\partial M$ is {\it order-detected} if $LO(M; [\alpha]) \ne \emptyset$. }
\end{definition}
For instance the slope corresponding to the rational longitude is always order-detected. See \cite[Example 6.3]{BGH1}.

Set 
$$\mathcal{D}_{ord}(M) = \{[\alpha] \in \mathcal{S}(M) \; | \; [\alpha] \hbox{ is order-detected}\}$$ 
and note that $\mathcal{D}_{ord}(M) \subseteq \mathcal{D}_{ord}^{wk}(M)$ follows immediately from the definitions. Moreover, if $\{ \mathfrak{o}_n \}$ is a sequence of orders in $\bigcup_{[\alpha] \in \mathcal{S}(M)} L(M; [\alpha])$ converging to $\mathfrak{o}$ and $g \in \pi_1(M)$, then  
\[ s(g \cdot \mathfrak{o}) = s(g \cdot \lim_n \mathfrak{o}_n) = \lim_n s(g \cdot \mathfrak{o}_n) = \lim_n s(\mathfrak{o}_n) = s(\mathfrak{o}).
\]

It follows that $\bigcup_{[\alpha] \in \mathcal{S}(M)} LO(M; [\alpha])$ is closed in the compact space $LO(M)$, so $\mathcal{D}_{ord}(M)$, its image in $\mathcal{S}(M)$ under $s$, is closed as well.  We record this as a proposition for future use.

\begin{proposition}
\label{prop:Dord is closed}
$\mathcal{D}_{ord}(M)$ is closed in $\mathcal{S}(M)$. 
\qed
\end{proposition}

Here is a restatement of Corollary \ref{cor: invariance under action}. 

\begin{proposition} 
\label{prop: cofinal 2}
Suppose that $\mathfrak{o}$ is a boundary-cofinal left-order on $\pi_1(M)$ which weakly order-detects a slope $[\alpha] \in \mathcal{S}(M)$. Then $\mathfrak{o} \in LO(M; [\alpha])$ and therefore $\mathfrak{o}$ order-detects $[\alpha]$. 
\qed
\end{proposition}

Referring to Example \ref{example: torus knot cofinality} we see that $\mathcal{D}_{ord}(M) = \mathcal{D}_{ord}^{wk}(M)$ when $M$ is a torus knot exterior. 
A similar conclusion follows from Theorem \ref{thm: some not weak implies all cofinal} and Proposition \ref{prop: cofinal 2} whenever $\mathcal{D}_{ord}^{wk}(M) \ne \mathcal{S}(M)$, thus proving Theorem \ref{thm: some not weak means weak detd = detd}. 

{\bf Theorem \ref{thm: some not weak means weak detd = detd}.}
{\it If $\mathcal{D}_{ord}^{wk}(M) \ne  \mathcal{S}(M)$, then $\mathcal{D}_{ord}(M) = \mathcal{D}_{ord}^{wk}(M)$. }
\qed

\subsection{Order-detection and gluing} 

Here we prove a gluing theorem for left-orders on knot manifolds and use it to provide an alternate 
characterisation of the order-detection of rational slopes (cf. Theorem \ref{thm: topological def of order-detection}). 

For a slope $[\alpha]$ on $\partial M$, let $LO(\partial M; [\alpha])$ denote the set of left-orders $\mathfrak{o}$ on $\pi_1(\partial M)$ for which $[L(\mathfrak{o})] = [\alpha]$ and recall from \S \ref{subsec: lo z2} that if $[\alpha]$ is irrational and $\mathfrak{o} \in LO(\partial M; [\alpha])$, then 
$$LO(\partial M; [\alpha]) = \{\mathfrak{o}, \mathfrak{o}^{op}\}$$
while if $[\alpha]$ is rational and $\mathfrak{o} \in LO(\partial M; [\alpha])$, there is a left-order $\mathfrak{o}^* \in LO(\partial M)$ such that 
$$LO(\partial M; [\alpha]) = \{\mathfrak{o}, \mathfrak{o}^{op}, \mathfrak{o}^*, (\mathfrak{o}^*)^{op}\}$$

\begin{lemma}
\label{lemma: det +}
Suppose that $\mathfrak{o}$ order-detects a slope $[\alpha]$. 

$(1)$ If $[\alpha]$ is an irrational slope and $\mathcal{N}$ is the normal family 
$\mathcal{O}(\mathfrak{o}) \cup \mathcal{O}(\mathfrak{o}^{op}) \subset LO(M; [\alpha]), $ 
then the restriction map $\mathcal{N} \to LO(\partial M; [\alpha])$ is surjective.

$(2)$ If $[\alpha]$ is a rational slope, there are left-orders $\hat{\mathfrak{o}}, \hat{\mathfrak{o}}' \in LO(M; [\alpha])$, where 
$\hat{\mathfrak{o}}\in \overline{\mathcal{O}(\mathfrak{o})}$, $\hat{\mathfrak{o}}'$ differs from $\hat{\mathfrak{o}}$ by a convex swap, and $\hat{\mathfrak{o}}'|_{\pi_1(\partial M)} = (\hat{\mathfrak{o}}|_{\pi_1(\partial M)})^*$. Consequently, if $\mathcal{N}$ is the normal family 
$\mathcal{O}(\hat{\mathfrak{o}})  \cup \mathcal{O}(\hat{\mathfrak{o}}')  \cup \mathcal{O}(\hat{\mathfrak{o}}^{op})  \cup  \mathcal{O}((\hat{\mathfrak{o}}')^{op}) \subset LO(M; [\alpha]),$ then 
the restriction map $\mathcal{N} \to LO(\partial M; [\alpha])$ 
is surjective.
\end{lemma}

\begin{proof}
Assertion (1) is obvious, so we need only deal with (2). 

Suppose that $[\alpha]$ is a rational slope and, without loss of generality, that $\alpha \in H_1(\partial M)$ is primitive. By Proposition \ref{prop: wk ord-det means convex} there is a left-order $\hat{\mathfrak{o}}  \in \overline{\mathcal{O}(\mathfrak{o})} \subset LO(M; [\alpha])$ and an $\hat{\mathfrak{o}}$-convex subgroup $C$ of $\pi_1(M)$ for which $C \cap \pi_1(\partial M) = \langle \alpha \rangle$. Since $C$ is $\hat{\mathfrak{o}}$-convex, we can apply Proposition \ref{prop: still convex}(2) to define $\hat{\mathfrak{o}}'$ by 
$$P(\hat{\mathfrak{o}}') = (P(\hat{\mathfrak{o}}) \cap C)^{-1} \cup (P(\hat{\mathfrak{o}}) \setminus C)$$
Then $C$ is $\hat{\mathfrak{o}}'$-convex and as $C \cap \pi_1(\partial M) = \langle \alpha \rangle$ we have 
$$\hat{\mathfrak{o}}'|_{\pi_1(\partial M)} = (\hat{\mathfrak{o}}|_{\pi_1(\partial M)})^*$$
(cf. \S \ref{subsec: lo z2}). To complete the proof it suffices to show that $\hat{\mathfrak{o}}' \in LO(M; [\alpha])$. That is, we must show $[L(\hat{\mathfrak{o}}'; g \pi_1(\partial M) g^{-1})]  = [\alpha]$ for each $g \in \pi_1(M)$ (Lemma \ref{lemma: variance under action}). 

If $g \in \pi_1(M)$, then $C \cap g \pi_1(\partial M) g^{-1}$ is a relatively convex subgroup of $g \pi_1(\partial M) g^{-1}$ and so we have three possibilities: 
 
\begin{itemize}

\item $C \cap g \pi_1(\partial M) g^{-1} = \{ 1 \}$ and therefore $P(\hat{\mathfrak{o}}') \cap g \pi_1(\partial M) g^{-1} = P(\hat{\mathfrak{o}}) \cap g \pi_1(\partial M) g^{-1}$;

\vspace{.2cm} \item $C \cap g \pi_1(\partial M) g^{-1} \cong \mathbb{Z}$ is a direct summand of $\pi_1(\partial M)$ and therefore, $C \cap g \pi_1(\partial M) g^{-1} = \langle g \alpha g^{-1}\rangle$ since $\hat{\mathfrak{o}}$ detects $[\alpha]$;

\vspace{.2cm} \item $C \cap g \pi_1(\partial M) g^{-1} = g \pi_1(\partial M) g^{-1}$ and therefore $P(\hat{\mathfrak{o}}') \cap g \pi_1(\partial M) g^{-1} = P(\hat{\mathfrak{o}})^{-1} \cap g \pi_1(\partial M) g^{-1} = P(\hat{\mathfrak{o}}^{op}) \cap g \pi_1(\partial M) g^{-1} $.

\end{itemize}
In the first of the three cases it is clear that $[L(\hat{\mathfrak{o}}'; g \pi_1(\partial M) g^{-1})]  = [L(\hat{\mathfrak{o}}; g \pi_1(\partial M) g^{-1})] = [\alpha]$, while in the third we have $[L(\hat{\mathfrak{o}}'; g \pi_1(\partial M) g^{-1})] = [L(\hat{\mathfrak{o}}^{op}; g \pi_1(\partial M) g^{-1})] = [L(\hat{\mathfrak{o}}; g \pi_1(M) g^{-1})] = [\alpha]$. Finally, in the second case we see that $\langle g \alpha g^{-1}\rangle$ is convex in $g \pi_1(\partial M) g^{-1}$ with respect to both $\hat{\mathfrak{o}}$ and $\hat{\mathfrak{o}}'$ and therefore $[L(\hat{\mathfrak{o}}'; g \pi_1(\partial M) g^{-1})]  = [L(\hat{\mathfrak{o}}; g \pi_1(\partial M) g^{-1})] = [\alpha]$. This completes the proof. 
\end{proof}

{\bf Theorem \ref{thm: gluing via detection}.}
{\it Suppose that $M_1$ and $M_2$ are knot manifolds and $W = M_1 \cup_f M_2$, where $f: \partial M_1 \xrightarrow{\; \cong \;} \partial M_2$ identifies slopes $[\alpha_1] \in \mathcal{D}_{ord}(M_1)$ on $\partial M_1$ and $[\alpha_2] \in \mathcal{D}_{ord}(M_2)$ on $\partial M_2$. Then $\pi_1(W)$ is left-orderable. 
}

\begin{proof}
According to the Bludov-Glass theorem \cite{BG}, $\pi_1(W)$ is left-orderable if and only if there are normal families of left-orderings $\mathcal{N}_1 \subset LO(M_1)$ and $\mathcal{N}_2 \subset LO(M_2)$ which are compatible with the gluing map $f: \partial M_1 \to \partial M_2$ in the sense that $f$ induces a bijection
\vspace{-.2cm} 
\begin{center} 
\begin{tikzpicture}[scale=0.8]
\node at (6, 4.5) {$\{ \mathfrak{o}|_{\pi_1(\partial M_1)}\; | \; \mathfrak{o} \in \mathcal{N}_1 \} $};
\node at (12, 4.5) {$\{ \mathfrak{o}|_{\pi_1(\partial M_2)} \; | \;\mathfrak{o} \in \mathcal{N}_2 \}$};

\node at (9, 5) {$f$}; 
\node at (9, 4.2) {$\cong$}; 

\draw [ <->] (8.05, 4.5) --(9.9,4.5); 

\end{tikzpicture}
\end{center} 
If we take $\mathcal{N}_1 \subset LO(M_1; [\alpha_1])$ and $\mathcal{N}_2 \subset LO(M_2; [\alpha_2])$ to be the normal families guaranteed by Lemma \ref{lemma: det +}, then the restriction maps $\mathcal{N}_i \to LO(\partial M_i; [\alpha_i])$, $i = 1, 2$, are surjective, and therefore as $f$ identifies $[\alpha_1] $ and $[\alpha_2]$,  the Bludov-Glass condition holds. (Note that $[\alpha_1]$ and $[\alpha_2]$ are either both rational or both irrational.) Thus $\pi_1(W)$ is left-orderable. 
\end{proof}

Here is a partial converse.  Given $3$-manifolds $M_1$ and $M_2$ with torus boundaries, and a homeomorphism $f: \partial M_1 \to \partial M_2$, we use $M_1 \cup_f M_2$ to denote the manifold obtained by attaching $M_1$ to $M_2$ via $f$.

{\bf Theorem \ref{thm: gluing via detection 2}}
{\it Suppose that $M_1$ and $M_2$ are knot manifolds such that every left-order on $\pi_1(M_2)$ is boundary-cofinal. Then $W = M_1 \cup_f M_2$ has a left-orderable fundamental group if and only if $f: \partial M_1 \xrightarrow{\; \cong \;} \partial M_2$ identifies slopes $[\alpha_1] \in \mathcal{D}_{ord}(M_1)$ on $\partial M_1$ and $[\alpha_2] \in \mathcal{D}_{ord}(M_2)$ on $\partial M_2$. 
}

\begin{proof}
The reverse direction is Theorem \ref{thm: gluing via detection}. For the forward direction, fix a left-order $\mathfrak{o} \in LO(W)$ and let $\mathfrak{o}_i$ be the restriction of $\mathfrak{o}$ to $\pi_1(M_i)$. Since $\mathfrak{o}_2$ is boundary-cofinal, so is $\mathfrak{o}_1$. Proposition \ref{prop: cofinal 2} then shows that $\mathfrak{o}_i \in LO(M_i, [L(\mathfrak{o}_i; \pi_1(\partial M_i))])$ and therefore $[L(\mathfrak{o}_i; \pi_1(\partial M_i))] \in \mathcal{D}_{ord}(M_i)$ for $i = 1, 2$. 
\end{proof}

Let $N$ be a twisted $I$-bundle over the Klein bottle with rational longitude $\lambda_N$.  The next theorem provides an alternative characterisation of order detection of rational slopes.

{\bf Theorem \ref{thm: topological def of order-detection}.}  
{\it Suppose that $M$ is boundary-irreducible, $[\alpha]$ is a rational slope on $\partial M$, and $f: \partial N \to \partial M$ is a homeomorphism which identifies $[\lambda_N]$ with $[\alpha]$. Then $[\alpha]$ is order-detected if and only if $\pi_1(M \cup_f N)$ is left-orderable. 
}

\begin{proof}
All left-orderings of $\pi_1(N)$ arise lexicographically from the short exact sequence
$$1 \rightarrow \langle y \rangle \rightarrow \pi_1(N) =  \langle x, y | xyx^{-1} = y^{-1} \rangle \rightarrow \langle x \rangle \rightarrow 1$$
(see \cite[Example 2.17]{CR} for instance) and so may be listed as $\{ \mathfrak{o},  \mathfrak{o}(x),  \mathfrak{o}(y),  \mathfrak{o}(x,y) \}$, where a generator appears in the parentheses if and only if it is positive with respect to that ordering. The element $y \in \pi_1(N)$ is peripheral and primitive, representing the rational longitude of $N$, and $\langle y \rangle$ is convex with respect to every order. Thus $[L({\mathfrak{o}'})] = [\lambda_{N}]$ for each left-order on $\pi_1(N)$ and the restriction map induces a bijection $LO(N) \to LO(\pi_1(\partial N); [\lambda_{N}])$. Theorem \ref{thm: gluing via detection} then shows that $\pi_1(M \cup_f N)$ is left-orderable. 

Conversely suppose that $M \cup_f N$ has a left-orderable fundamental group and let $\mathcal{N}_1, \mathcal{N}_2$ be normal families of left-orders as guaranteed by the Bludov-Glass theorem. As noted above, each left-order in $\mathcal{N}_2$ detects $f_*([\lambda_{N}]) = [\alpha]$ and therefore $[L({g \cdot \mathfrak{o}})] = [\alpha]$ for each $\mathfrak{o} \in \mathcal{N}_1$ and $g \in \pi_1(M)$. Thus $[\alpha] \in \mathcal{D}_{ord}(M)$. 
\end{proof}

\begin{remark}
\label{rem: logst}
We could replace $N$ in this theorem by any $LO$-generalised solid torus. In other words by a knot manifold $M$ whose rational longitude is the only order-detected slope on $\partial M$. Examples include the knot manifolds $N_t$ of \cite[\S 2.2.3]{BC} ($t \geq 2$), any Seifert fibred knot manifold with base orbifold the Mobius band with cone points \cite[Proposition 4.7]{BC}, and the hyperbolic knot manifold $v2503$ (\cite[Proposition 7.2]{BGH1}).
\end{remark}

\subsection{Extensions of the gluing theorem} 
\label{subsec: general gluings} 

The results of the previous section can be extended in several useful ways which require only minor adjustments to the proofs. 

First note that if $T$ is an incompressible torus in the boundary of a $3$-manifold $M$, the restriction of each left-order $\mathfrak{o} \in LO(M)$ to $\pi_1(T)$ weakly detects (i.e. determines) a slope on $T$ as in \S \ref{subsec: wk-ord-det}. The set of all such slopes is denoted by $\mathcal{D}_{ord}^{wk}(M, T)$, while $\mathcal{D}_{ord}(M, T)$ denotes the set of slopes $[\alpha]$ on $T$ for which there is some $\mathfrak{o} \in LO(M)$ such that $\gamma \cdot \mathfrak{o}$ weakly detects $[\alpha]$ for all $\gamma \in \pi_1(M)$. 

The proof of Theorem \ref{thm: gluing via detection} is easily modified to this more general situation: Suppose that $M_1$ and $M_2$ are $3$-manifolds and $T_1 \subseteq \partial M_1, T_2 \subseteq \partial M_2$ are incompressible tori. If $W = M_1 \cup_f M_2$, where $f: T_1 \xrightarrow{\; \cong \;} T_2$ identifies slopes $[\alpha_1] \in \mathcal{D}_{ord}(M_1, T_1)$ on $\partial M_1$ and $[\alpha_2] \in \mathcal{D}_{ord}(M_2, T_2)$ on $\partial M_2$, then $\pi_1(W)$ is left-orderable. With applications in mind, we present a more refined version of this result. To state it, we introduce multislopes and their order-detection.

Suppose that $M$ is a compact, connected, orientable, irreducible $3$-manifold whose boundary is a union $\partial M = T_1 \sqcup T_2 \sqcup \cdots \sqcup T_r$ of incompressible tori. Each $\mathfrak{o} \in LO(M)$ determines a multislope 
$$[L(\mathfrak{o})] = ([L(\mathfrak{o}|_{\pi_1(T_1)})], \ldots , [L(\mathfrak{o}|_{\pi_1(T_r)})]) \in \mathcal{S}(M) = \mathcal{S}(T_1) \times \mathcal{S}(T_2) \times \cdots \times \mathcal{S}(T_r) \cong (S^1)^r.$$
We say that $\mathfrak{o}$ weakly order-detects $[L(\mathfrak{o})]$ and denote the set of weakly order-detected multislopes by $\mathcal{D}_{ord}^{wk}(M)$. We say that $\mathfrak{o}$ order-detects $[L(\mathfrak{o})]$ if $\gamma \cdot \mathfrak{o}$ weakly order-detects it for each $\gamma \in \pi_1(M)$, and denote the set of order-detected multislopes by $\mathcal{D}_{ord}(M)$.

\begin{theorem}
\label{thm: general gluing}
Suppose that $M_1$ and $M_2$ are $3$-manifolds and that $\partial M_i$ is a union of incompressible tori $T_{i1} \sqcup T_{i2} \sqcup \ldots \sqcup T_{ir_i}$. Suppose, moreover, that $\mathfrak{o}_i$ is a left-order on $\pi_1(M_i)$ which order-detects a multislope $([\alpha_{i1}], [\alpha_{i2}], \ldots, [\alpha_{ir_i}])$, and that $f: T_{11} \to T_{21}$ is a homeomorphism which identifies $[\alpha_{11}]$ with $[\alpha_{21}]$. Then there is a left-order on $\pi_1(M_1 \cup_f M_2)$ which order-detects the multislope $([\alpha_{12}], [\alpha_{13}], \ldots, [\alpha_{1r_1}], [\alpha_{22}], [\alpha_{23}], \ldots, [\alpha_{2r_2}])$ on $\partial(M_1 \cup_f M_2)$. 
\end{theorem}

\begin{proof}
The proof proceeds as in that of Theorem \ref{thm: gluing via detection}, though with a more careful application of the Bludov-Glass theorem. 

The result is obvious if $M_i$ is a product on $T_{i1}$, for in this case $r_i = 2$ and $[\alpha_{i2}]$ is identified with $[\alpha_{i1}]$ under the identifications $T_{ij} \equiv T_{i1} \times \{j\}$. Assume then that neither $M_i$ is a product on $T_{i1}$. Equivalently, $\pi_1(T_{i1})$ is a proper subgroup of $\pi_1(M_i)$ for both $i$, which implies that if conjugation by $\gamma \in \pi_1(M_1 \cup_f M_2)$ stabilises $\pi_1(M_i)$, then $\gamma \in \pi_1(M_i)$.  

 If $[\alpha_{11}]$ is an irrational slope, the Bludov-Glass theorem is applied (in the proof of Theorem \ref{thm: gluing via detection}) with respect to the normal families $\mathcal{N}_i = \mathcal{O}(\mathfrak{o}_i) \cup \mathcal{O}(\mathfrak{o}_i^{op})$ on $\pi_1(M_i)$, $i = 1, 2$. Then for each $\gamma \in \pi_1(M_1 \cup_f M_2)$, the resulting left-order $\mathfrak{o}$ of $\pi_1(M_1 \cup_f M_2)$ restricts to a left-order on each conjugate $\gamma \pi_1(M_i) \gamma^{-1}$  of $\pi_1(M_i)$ in $\pi_1(M_1 \cup_f M_2)$ which corresponds to a left-order in $\mathcal{N}_i$ under the isomorphism $\pi_1(M_i) \to \gamma \pi_1(M_i) \gamma^{-1}, x \mapsto \gamma x \gamma^{-1}$ (see \cite[Proof of Theorem A]{BG}\footnote{In particular, this fact is implied by the sentence ``By the construction, each $\alpha$-order on $L$ ($\alpha \in \Gamma^\#$) induces orders on $G_1$ and $G_2$ belonging to $\mathcal{R}_1$ and $\mathcal{R}_2$, respectively" appearing in the proof of Theorem A on page 596.}). 
 
 As this isomorphism is well-defined up to conjugation in $\pi_1(M_i)$ and each element of $\mathcal{N}_i$ order-detects $([\alpha_{i1}], [\alpha_{i2}], \ldots, [\alpha_{ir_i}])$,  $\mathfrak{o}$ order-detects $([\alpha_{12}], [\alpha_{13}], \ldots, [\alpha_{1r_1}], [\alpha_{22}], \ldots, [\alpha_{2r_2}])$, as claimed.  

A similar argument holds when $[\alpha_{11}]$ is a rational slope, where the Bludov-Glass theorem is applied with respect to normal families 
$\mathcal{N}_i = \mathcal{O}(\hat{\mathfrak{o}}_i)  \cup \mathcal{O}(\hat{\mathfrak{o}}_i')  \cup \mathcal{O}(\hat{\mathfrak{o}}_i^{op})  \cup  \mathcal{O}((\hat{\mathfrak{o}}_i')^{op})$, $\hat{\mathfrak{o}}_i \in \overline{\mathcal{O}(\mathfrak{o}_i)}$, $\hat{\mathfrak{o}}_i'$ differs from $\hat{\mathfrak{o}}_i$ by a convex swap, and $\hat{\mathfrak{o}}_i'|_{\pi_1(T_{i1})} = (\hat{\mathfrak{o}}_i|_{\pi_1(T_{i1})})^*$.  We first observe that the orderings $\hat{\mathfrak{o}}_i, \hat{\mathfrak{o}}_i'$ provided to us by Lemma \ref{lemma: det +} order-detect  $([\alpha_{i1}], [\alpha_{i2}], \ldots, [\alpha_{ir_i}])$, and the same holds for $\hat{\mathfrak{o}}_i^{op}, (\hat{\mathfrak{o}}_i')^{op}$.  

In the proof of Lemma \ref{lemma: det +}, we first arrive at $\hat{\mathfrak{o}}_i$ from $\mathfrak{o}_i$ by an application of Proposition \ref{prop: wk ord-det means convex}, which implies that $\hat{\mathfrak{o}}_i \in \overline{\mathcal{O}(\mathfrak{o}_i)}$.   The slope map $s_{ij} : LO(M_i) \rightarrow \mathcal{S}_j(M_i)$, where $\mathcal{S}_j(M_i) = H_1(T_{ij}; \mathbb{R}) / \{ \pm 1\}$, is continuous and constant on $\mathcal{O}(\mathfrak{o}_i)$ by assumption.  Moreover $\mathcal{O}(\mathfrak{o}_i)$ is invariant under the action of $\pi_1(M_i)$ on $LO(M_i)$, from which we conclude that $s_{ij}(g \cdot \hat{\mathfrak{o}}_i) = s_{ij}(\mathfrak{o}_i)$ for all $g \in \pi_1(M_i)$.  Thus $ \hat{\mathfrak{o}}_i$ order-detects $([\alpha_{i1}], [\alpha_{i2}], \ldots, [\alpha_{ir_i}])$.   To show the same is true of $\hat{\mathfrak{o}}_i'$ we can proceed as in the proof of Lemma \ref{lemma: det +}.

From here, the reasoning in the rational case of the proof of Theorem \ref{thm: gluing via detection} applies, and as in the previous paragraph, this implies that the order $\mathfrak{o}$ arising from an application of the Bludov-Glass theorem order-detects $([\alpha_{12}], [\alpha_{13}], \ldots, [\alpha_{1r_1}], [\alpha_{22}], \ldots, [\alpha_{2r_2}])$. 
\end{proof}

\begin{proposition} 
\label{prop: characterisation of order-detection}
Suppose that $M$ is a compact, connected, orientable, boundary-incompressible $3$-manifold whose boundary is a non-empty union of incompressible tori $T_1, T_2, \ldots, T_m$ and consider a multislope $[\alpha] =  ([\alpha_1], [\alpha_2], \ldots, [\alpha_m]) \in \mathcal{S}(M)$ where $[\alpha_i]$ is rational for $1 \leq i \leq r$. Set $W = M \cup_{T_1} N_1 \cup_{T_2} \cdots \cup_{T_r} N_r$ where each $N_i$ is a copy of the twisted $I$-bundle over the Klein bottle glued to $M$ along $T_i$ so that $[\lambda_{N_i}]$ is identified with $[\alpha_i]$.   

$(1)$ For $r < m$, $[\alpha] \in \mathcal{D}_{ord}(M)$ if and only if $([\alpha_{r+1}], \ldots, [\alpha_m]) \in \mathcal{S}(W)$. 

$(2)$ For $r = m$, $[\alpha] \in \mathcal{D}_{ord}(M)$ if and only if $\pi_1(W)$ is left-orderable. 
\end{proposition}

\begin{proof}
The proof is analogous to that of Theorem \ref{thm: topological def of order-detection}. In the forward direction we replace the use of Theorem \ref{thm: gluing via detection} there with an inductive application of Theorem \ref{thm: general gluing} here. The proof in the reverse direction is entirely analogous. 
\end{proof}

We could replace any of the copies of $N$ in this result by an arbitrary $LO$-generalised solid torus (cf. Remark \ref{rem: logst}). 

We finish this section with a result which gives sufficient conditions for the left-orderability of the fundamental group of a closed $3$-manifold split into pieces along incompressible tori. A result of this nature first appeared in \cite{BC} in the case of graph manifolds, while the one we state below is Theorem 7.6 of \cite{BGH1}. For completeness, we provide its proof, which is a simple consequence of Theorem \ref{thm: gluing via detection} and Proposition \ref{prop: characterisation of order-detection}.

Let $W$ be closed, connected, irreducible rational homology $3$-sphere and $T_1, T_2, \ldots, T_m$ a disjoint family of essential tori in $W$ which split it into pieces $M_1, M_2, \ldots, M_n$. For $1 \leq j \leq n$ write $\partial M_j = T_{i_1} \sqcup T_{i_2} \sqcup \cdots \sqcup T_{i_{r(j)}}$. For each family of slopes $([\alpha_1], [\alpha_2], \ldots, [\alpha_m]) \in \mathcal{S}(T_1) \times \mathcal{S}(T_2) \times \cdots \times \mathcal{S}(T_m)$, let 
$$[L(j)] = ([\alpha_{i_1}], [\alpha_{i_2}], \ldots, [\alpha_{i_{r(j)}}]) \in \mathcal{S}(M_j)$$ 

Here is the generalised gluing theorem of \cite{BGH1}. 

\begin{definition}
We call $([\alpha_1], [\alpha_2], \ldots, [\alpha_m]) \in \mathcal{S}(T_1) \times \mathcal{S}(T_2) \times \cdots \times \mathcal{S}(T_m)$ {\it gluing coherent} if $[L(j)] \in \mathcal{D}_{ord}(M_j)$ for each $j$. 
\end{definition}

\begin{theorem} {\rm (Boyer-Gordon-Hu)}
\label{thm: general * gluing}
Let $W = \cup_j M_j$ be closed, connected, irreducible $3$-manifold expressed as a union of submanifolds $M_1, M_2, \ldots, M_n$ along a disjoint family of essential tori $T_1, T_2, \ldots, T_m$. If there is a gluing coherent family of slopes $([\alpha_1], [\alpha_2], \ldots, [\alpha_m]) \in \mathcal{S}(T_1) \times \mathcal{S}(T_2) \times \cdots \times \mathcal{S}(T_m)$, then $W$ has a left-orderable fundamental group. 
\end{theorem}

\begin{proof}
Without loss of generality we can suppose that $W$ is a rational homology $3$-sphere (\cite[Theorem 1.1]{BRW}), so that each $T_i$ is separating. We induct on $m$, the case $m = 1$ being Theorem \ref{thm: gluing via detection}. Suppose then that $m \geq 2$.

After re-indexing we can suppose that $M_1, \ldots, M_r$ lie to one side of $T_1$, $M_{r+1}, \ldots, M_n$ to the other, and $M_r \cap M_{r+1} = T_1$. Set $M' = M_1 \cup \cdots \cup M_r$ and $M'' = M_{r+1} \cup \cdots \cup M_n$, so that $W = M' \cup_{T_1} M''$. We can assume that $T_2, \ldots, T_s$ are contained in the interior of $M'$ and $T_{s+1}, \ldots, T_m$ in the interior of $M''$. 

For $j = r, r+1$ define $M_j' = M_j \cup N$, where $N$ is a twisted $I$-bundle over the Klein bottle which is attached to $M_j$ by a homeomorphism which identifies the longitudinal slope of $N$ with $[\alpha_1]$. An application of Proposition \ref{prop: characterisation of order-detection} then shows that $([\alpha_{2}], \ldots, [\alpha_s])$ is gluing coherent in the closed manifold $W_1 = M_1 \cup \cdots \cup M_{r-1} \cup M_r'$, so our inductive hypothesis implies that $\pi_1(W_1)$ is left-orderable. Another application of Proposition \ref{prop: characterisation of order-detection} implies that $[\alpha_1]$ is order-detected on $\partial (M_1 \cup \cdots \cup M_r)$. Similarly, $[\alpha_1]$ is order-detected on $\partial (M_{r+1} \cup \cdots \cup M_n)$. The case $m = 1$ now implies that $W$ has a left-orderable fundamental group. 
\end{proof}

\section{Strong order-detection of slopes}
\label{sec: str}
The notion of order-detection introduced in the previous section was designed to understand the left-orderability of the fundamental group of a union of two knot manifolds $M_1, M_2$ glued along their (incompressible) boundary tori. Here we introduce a form of order-detection adapted to the situation where $M_2$ is a solid torus. In other words, to Dehn fillings of $M_1$. 

\begin{definition}  
\label{def: str det}
Let $[\alpha]$ be a slope on $\partial M$.  We say that $[\alpha]$ is {\it strongly order-detected} if it is an irrational slope and order-detected or it is a rational slope and there is a left-order $\mathfrak{o} \in LO(M)$ and a proper $\mathfrak{o}$-convex, normal subgroup $C$ of $\pi_1(M)$ such that $\langle \alpha \rangle \leq C \cap \pi_1(\partial M)$.   
\end{definition}
Set 
$$\mathcal{D}_{ord}^{str}(M) = \{[\alpha] \in \mathcal{S}(M) \; | \; [\alpha] \hbox{ is strongly order-detected}\}$$

Though the definition of strong order-detection is intrinsic to $LO(M_1)$,  in the case of rational slopes it is essentially equivalent to the left-orderability of the associated Dehn filling of $M_1$. It is interesting to compare this situation with Theorem \ref{thm: topological def of order-detection}. 

{\bf Theorem \ref{thm: str det and df quot}.} 
{\it Suppose that $[\alpha]$ is a rational slope on $\partial M$. Then $[\alpha] \in \mathcal{D}_{ord}^{str}(M)$ if and only if $\pi_1(M(\alpha))$ has a left-orderable quotient. }

\begin{proof}
If $[\alpha] \in \mathcal{D}_{ord}^{str}(M)$ there is a proper $\mathfrak{o}$-convex, normal subgroup $C$ of $\pi_1(M)$ such that $\langle \alpha \rangle \leq C \cap \pi_1(\partial M)$ and therefore we have an epimorphism from $\pi_1(M(\alpha) $ the left-orderable group $\pi_1(M)/C$. 

Conversely suppose that there is an epimorphism $\pi_1(M(\alpha)) \to G$, where $G$ is left-orderable and let $C$ be the kernel of the composition $\pi_1(M) \to \pi_1(M(\alpha)) \to G$, so $\langle \alpha \rangle \leq C \cap \pi_1(\partial M)$. As a subgroup of a left-orderable group, $C$ is left-orderable, so we can use the exact sequence
$$1 \to C \to \pi_1(M) \to G$$
to find a left-order $\mathfrak{o}$ on $\pi_1(M)$ for which $C$ is $\mathfrak{o}$-convex, which completes the proof. 
\end{proof}

{\bf Corollary \ref {cor: str detect and dehn filling}.} 
{\it Suppose that $[\alpha]$ is a rational slope on $\partial M$. If $\pi_1(M(\alpha))$ is left-orderable, then $[\alpha]$ is strongly order-detected.  Conversely, 
if $[\alpha]$ is strongly order-detected and $M(\alpha)$ is irreducible, then $\pi_1(M(\alpha))$ is left-orderable.}
\qed

\begin{proposition} 
\label{prop: most str detected}
Suppose that $M$ is a knot manifold, $\mathfrak{o} \in LO(M)$, and $C$ is a proper $\mathfrak{o}$-convex, normal subgroup $C$ of $\pi_1(M)$ which contains $\pi_1(\partial M)$. Then,

$(1)$ $\mathcal{S}(M) \setminus \mathcal{D}_{ord}^{str}(M) \subseteq \{[\alpha] \in \mathcal{S}_{rat}(M) \; | \; M(\alpha) \mbox{ is reducible}\}$. Hence $\#(\mathcal{S}(M) \setminus \mathcal{D}_{ord}^{str}(M)) \leq 3$.

$(2)$ $\mathcal{D}_{ord}(M) = \mathcal{S}(M)$.
\end{proposition}

\begin{proof}
Part (1) is an immediate consequence of Theorem \ref{thm: str det and df quot}.

As a prelude to proving part (2), let 
$$K(M) = \{\mbox{primitive } \beta^{\pm 1} \in \pi_1(\partial M) \; | \; \pi_1(\partial M)  \leq \langle \langle \beta \rangle \rangle\},$$
where $\langle \langle \beta \rangle \rangle$ denotes the normal closure in $\pi_1(M)$ of $\beta$, and 
$$K^*(M) = \{\mbox{primitive } \beta \in \pi_1(\partial M) \; | \; \beta \in K(M) \mbox{ or } M(\beta) \mbox{ is reducible}\}$$
We expect $K^*(M)$ to be small and know of no knot manifold $M$ where $\#K^*\hspace{-.4mm}(M) > 3$. Simple considerations show that 
$K(M) \subset \{\mbox{primitive } \beta \in \pi_1(\partial M) \; | \;   |\beta \cdot \lambda_M| = 1\}$, a discrete subset of $\mathcal{S}(M)$.

By hypothesis, any rational slope $[\alpha] \in \mathcal{S}(M) \setminus K^*\hspace{-.4mm}(M)$ is strongly order-detected rational slope, and as $M(\alpha)$ is irreducible, Corollary \ref {cor: str detect and dehn filling} implies that $\pi_1(M(\alpha))$ is left-orderable. 

Let $\mathfrak{o}'$ to be any left-order on $\pi_1(M)$ obtained from the short exact sequence
$$1 \to \langle \langle \alpha \rangle \rangle \to \pi_1(M) \to \pi_1(M(\alpha)) \to 1$$ 
Then $\langle \langle \alpha \rangle \rangle$ is $\mathfrak{o}'$-convex, so $\pi_1(\partial M)  \cap \langle \langle \alpha \rangle \rangle$ is a primitive subgroup of $\pi_1(\partial M)$. But then as $\alpha \not \in K(M)$, $\pi_1(\partial M)  \cap \langle \langle \alpha \rangle \rangle = \langle \alpha \rangle$. Thus $\mathfrak{o}'$ weakly order-detects $[\alpha]$. Moreover if $g \in \pi_1(M)$,  
$$\langle \langle \alpha \rangle \rangle \cap g \pi_1(\partial M) g^{-1} = g\big(\langle \langle \alpha \rangle \rangle \cap  \pi_1(\partial M)\big)g^{-1}  = g \langle \alpha  \rangle g^{-1}= \langle g \alpha g^{-1} \rangle$$  
Thus $\mathfrak{o}'$ order-detects $[\alpha]$ and consequently, $\mathcal{S}_{rat}(M) \setminus K^*\hspace{-.4mm}(M) \subset \mathcal{D}_{ord}(M)$. Then as $\mathcal{S}_{rat}(M) \setminus K^*\hspace{-.4mm}(M)$ is dense in $\mathcal{S}(M)$ and $\mathcal{D}_{ord}(M)$ is closed, $\mathcal{D}_{ord}(M) = \mathcal{S}(M)$, which is (2).   
\end{proof}
 
\begin{corollary} 
\label{cor: str implies reg}
$\mathcal{D}_{ord}^{str}(M) \subseteq \mathcal{D}_{ord}(M)$.
\end{corollary}

\begin{proof} 
Let $[\alpha] \in \mathcal{D}_{ord}^{str}(M)$. If it is irrational, it is order-detected by definition. If it is rational, there is a left-order $\mathfrak{o} \in LO(M)$ and a proper $\mathfrak{o}$-convex, normal subgroup $C$ of $\pi_1(M)$ such that $\langle \alpha \rangle \leq C \cap \pi_1(\partial M)$. If $\pi_1(\partial M) \leq C$ then the previous proposition shows that $\mathcal{D}_{ord}(M) = \mathcal{S}(M)$, which completes the proof. Otherwise $C \cap \pi_1(\partial M) = \langle \alpha \rangle$ and therefore the normality of $C$ implies that $\langle \langle \alpha \rangle \rangle \leq C$. Then
$$C \cap g \pi_1(\partial M) g^{-1} = g\big(C \cap  \pi_1(\partial M)\big)g^{-1}  = g \langle \alpha  \rangle g^{-1}= \langle g \alpha g^{-1} \rangle,$$  
so $[\alpha]$ is order-detected. 
\end{proof} 

It is possible that $\mathcal{D}_{ord}^{str}(M) = \mathcal{D}_{ord}(M)$. For instance this occurs in the extreme cases that $\mathcal{D}_{ord}^{str}(M) = \mathcal{S}(M)$ (e.g. $b_1(M) \geq 2$) or $\mathcal{D}_{ord}(M) = \{\lambda_M\}$. The latter occurs for any Seifert fibre space with base orbifold a Mobius band with cone points (e.g. the twisted $I$-bundle over the Klein bottle), as is shown in \cite{BC}. However, the next example illustrates that in general, the containment $\mathcal{D}_{ord}^{str}(M) \subseteq \mathcal{D}_{ord}(M)$ may be proper.

\begin{example}
Suppose that $M$ is the complement of a $(r, s)$ torus knot $T_{r,s} \subset S^3$ ($r,s \geq 2$) and denote its standard meridian and longitude basis for $\pi_1(\partial M)$ by $\{ \mu, \lambda\}$. It follows from the work of Jankins and Neumann \cite{JN} and of Naimi \cite{Na} that under the usual identification of $\mathcal{S}(M)$ with $\mathbb R \cup \{\infty\}$, $[x\mu + y\lambda] \leftrightarrow x/y$, 
$$\mathcal{D}_{ord}^{str}(M) = (-\infty, rs - r - s) \subset [-\infty, rs - r - s] = \mathcal{D}_{ord}(M)$$ 
(See \cite[Appendix A]{BC}.) In particular, $\{ [\mu], [\mu^{rs-r-s} \lambda] \} \subset  \mathcal{D}_{ord}(M) \setminus \mathcal{D}_{ord}^{str}(M)$. 
  \end{example}

\end{document}